\documentclass[review,onefignum,onetabnum]{siamart220329}

\usepackage{amsfonts,comment}
\usepackage{amssymb,amsmath,url,bm}

\usepackage{graphicx,color,float,epstopdf}
\usepackage{mathtools}
\usepackage{pst-blur,pstricks-add}
\usepackage{cases}
\usepackage{algorithm}
\usepackage{algorithmicx}
\usepackage{algpseudocode}
\usepackage{multirow}
\usepackage{amsopn}
\usepackage{mathabx}

\nolinenumbers

\usepackage{geometry}
\DeclareMathOperator{\prox}{Prox}

\newtheorem{ass}[theorem]{Assumption}
\newtheorem{example}[theorem]{Example}

\newcommand{\h}[1]{\mathbf{#1}}

\DeclareMathOperator{\Rbb}{\mathbb{R}}
\DeclareMathOperator{\R}{\mathbb{R}}
\DeclareMathOperator*{\proj}{Proj}
\DeclareMathOperator*{\dom}{dom}
\DeclareMathOperator*{\inte}{int}
\DeclareMathOperator*{\argmin}{argmin}

\newcommand{\nn}{\nonumber}

\newcommand{\C}{\mathbb{C}}

\newcommand{\norm}[1]{\left \| #1 \right \|}

\newcommand{\skal}[1]{\left \langle #1 \right \rangle}
\DeclareMathOperator{\dist}{dist}

\ifpdf
  \DeclareGraphicsExtensions{.eps,.pdf,.png,.jpg}
\else
  \DeclareGraphicsExtensions{.eps}
\fi

\newsiamremark{remark}{Remark}
\newsiamremark{hypothesis}{Hypothesis}
\crefname{hypothesis}{Hypothesis}{Hypotheses}
\newsiamthm{claim}{Claim}

\headers{Full splitting algorithm for structured DC}{Radu Ioan Bo\c t, Rossen Nenov and Min Tao}

\title{A full splitting algorithm for structured difference-of-convex programs \thanks{Submitted to editors DATE.}
\funding{The research of RIB and RN has been partially supported by the Austrian Science Fund (FWF), 10.55776/W1260 and 10.55776/P34922. The research of MT has been partially supported by the Natural Science Foundation of China (No.
12471289).}}

\author{Radu Ioan Bo\c t \thanks{Faculty of Mathematics, University of Vienna, A-1090 Vienna, Austria, \email{radu.bot@univie.ac.at}.}
\and
{Rossen Nenov} \thanks{Acoustic Research Institute, Austrian Academy of Sciences, A-1010 Vienna, Austria, \email{rossen.nenov@oeaw.ac.at}.}
\and
{Min Tao} \thanks{School of Mathematics, National Key Laboratory for Novel Software Technology, Nanjing University, Nanjing, 210093, Republic of China, \email{taom@nju.edu.cn}}.}

\begin{document}

\maketitle
\begin{abstract}
In this paper, we study a class of nonconvex and nonsmooth structured difference-of-convex (DC) programs, which contain in the convex part the sum of a nonsmooth linearly composed convex function and a differentiable function, and in the concave part another nonsmooth linearly composed convex function. Among the various areas in which such problems occur, we would like to mention in particular the recovery of sparse signals. We propose an adaptive double-proximal, full-splitting algorithm with a moving center approach in the final subproblem, which addresses the challenge of evaluating compositions by decoupling the linear operator from the nonsmooth component. We establish the subsequential convergence of the generated sequence of iterates to an approximate stationary point and prove its global convergence under the Kurdyka-\L ojasiewicz property. We also discuss the tightness of the convergence results and provide insights into the rationale for seeking an approximate KKT point. This is illustrated by constructing a counterexample showing that the algorithm can diverge when seeking exact solutions. Finally, we present a practical version of the algorithm that incorporates a nonmonotone line search, which significantly improves the convergence performance.
\end{abstract}



\section{Introduction}\label{sec:Intro}

In this paper, we consider the following nonsmooth difference-of-convex (DC) program
\begin{eqnarray}\label{PForm} \min_{{\h x} \in S} F({\h x}):=g(A{\h x})+\varphi({\h x})-f(K{\h x}), \end{eqnarray}
where $ S \subseteq \R^n$ is a nonempty, convex and compact set,  $f:{\mathbb R}^p\rightarrow{\overline{\mathbb R}}:={\mathbb R}\cup\{ {\pm}\infty\}$ and $g: {\mathbb R}^s\rightarrow{\overline{\mathbb R}}$ are proper, convex and  lower semicontinuous functions,  $A:{\mathbb R}^{n}\rightarrow{\mathbb R}^s$ and  $K: {\mathbb R}^{n}\rightarrow {\mathbb R}^p$ are linear operators, and $\varphi: {\mathbb R}^n \rightarrow {\mathbb R}$ is a differentiable function with $\ell_{\nabla \varphi}$-Lipschitz continuous gradient.

The study of this problem has attracted considerable interest in recent years due to its wide range of applications in various fields, most notably image processing \cite{App:Im}, machine learning \cite{App:MachineLearning}, optimal transport \cite{App:Optimal}, cryptography \cite{App:Crypto} and sparse signal recovery \cite{App:Sparse}. A more detailed overview of applications can be found in \cite{30yearsOverview}.

DC programming is particularly challenging due to the nonsmooth and nonconvex nature of the objective function.  One of the best known iterative methods for solving DC problems is the DC algorithm (DCA) introduced in \cite{TaoAn}. DCA linearizes the concave part of the objective function using subgradients and solves the resulting convex subproblem in each iteration.

In general, DCA poses two main challenges. First, the resulting convex subproblem can be difficult to solve, especially in practical applications where the problem structure is complex or high-dimensional. In such cases, it may be necessary to use a specialized non-smooth nonlinear optimization solver to effectively deal with the subproblems that arise during DCA iterations. Secondly, it relies entirely on subgradients of the concave part, which can be disadvantageous due to the possible nonuniqueness of the subgradients and the highly discontinuous nature of the subdifferential.

Developments in DC programming have focused on enhancing both the efficiency and convergence properties of DCA. One advancement in this regard is the evaluation of the convex component of the objective function by its proximal point operator, which has been shown to improve the stability and the convergence rates of the algorithm \cite{Sun03}. Additionally, inexact proximal Newton methods have been applied to DC programming in \cite{nakayama2024inexact}.
Another notable contribution is the development of the refined inertial DC algorithm \cite{you2023refined}, which integrates inertia into the optimization process to accelerate convergence while maintaining stability.  Furthermore, stochastic variants of DCA have been explored to tackle large-scale problems where deterministic approaches become computationally prohibitive \cite{stochasticDCA}. Another recent direction is the use of Armijo's backtracking linesearch rule \cite{ armijo1966minimization} to improve convergence. Work in this direction can be found in \cite{artacho2017accelerating}, which proposed the so-called Boosted DCA (BDCA) method for minimizing unconstrained smooth DC functions, and in \cite{artacho2019boosted}, which extended the former work for some classes of nonsmooth DC functions, and in \cite{artacho2022boosted} for linearly constrained DC programs.

In \cite{joki2018double,gaudioso2018minimizing,oliveira2019proximal}, a family of DC bundle methods has been developed based on the idea of replacing the convex subproblems in DCA at each iteration by strictly convex quadratic programs. Another notable approach is Sequential Difference-of-Convex Programming (SDCP) \cite{Oliveira20}, which replaces the first component with a convex model tailored to the feasible set and objective structure, while directly handling the concave part to simplify the computation of critical points.

The work \cite{SB19} represents a significant advance in this field by proposing a double-proximal gradient algorithm that evaluates both the convex and concave parts of the objective function using their respective proximal point operators. These evaluations differ from traditional methods, as they do not need to linearize either function. Furthermore, the inclusion of a linear operator in the concave part has been considered and evaluated in a forward manner similar to primal-dual full splitting methods.

This paper aims to extend this work by including a linear operator in the convex part of the objective function and developing a novel full splitting scheme. It relies only on gradient evaluation of the smooth component, proximal evaluation of the convex and concave parts separately, and conjugate unfolding to evaluate the linear operators by forward evaluation.

This paper is structured as follows. Section \ref{sec:Preliminaries} introduces the necessary preliminary definitions and results, while Section \ref{sec:Assumptions} states the main assumptions and notions of stationarity. Section \ref{sec:DPFS} introduces the conceptual full splitting scheme and discusses adaptive step size choices. Subsection \ref{sec:Adaptive_DPFS} introduces the Adaptive DPFS$_{mc}$ algorithm and presents its asymptotic convergence results. Section \ref{sec:GlobalConv} discusses the global convergence of the proposed algorithm under the assumption of the Kurdyka-\L{}ojasiewicz property, and Section \ref{sec:ConvRates} proves convergence rates and theoretical results for functions with Kurdyka-\L{}ojasiewicz exponents. Finally, in Section \ref{sec:Application} we perform numerical experiments demonstrating the effectiveness of our method for audio signal denoising.

\section{Preliminaries} \label{sec:Preliminaries}

In this section, we will introduce some notions and preliminary results that will be utilized throughout the paper. Finite-dimensional spaces within the paper will be equipped with the Euclidean norm, denoted by $\|\cdot\|$, while $\langle \cdot, \cdot \rangle$ will represent the Euclidean scalar product. For ${\h x} \in {\mathbb R}^n$ and $r >0$, $B({\h x},r):=\{{\h y} \in {\mathbb R}^n : \|{\h y} - {\h x}\| < r\}$ denotes the {\it open ball} with center ${\h x}$ and radius $r$. Given a set ${\cal C} \subseteq {\mathbb R}^n$,  ${\text{ri}}({\cal C})$ and ${\text{int}}({\cal C})$ denote its {\it relative interior} and its {\it interior}, respectively. For ${\h x} \in {\mathbb R}^n$, we denote by ${\text{\rm dist}}({\h x},{\cal C}) := \inf \{\|{\h x}-{\h c}\| : {\h c} \in {\cal C}\}$ the {\it distance} from $x$ to ${\cal C}$. If ${\cal C}$ is nonempty, convex and closed, the unique element at which the infimum is attained is the {\it projection} of ${\h x}$ onto ${\cal C}$, which we denote by $\proj_{\cal C}({\h x})$. The function $\iota_{\cal C} : {\mathbb R}^n \rightarrow {\overline{\mathbb R}}$, defined by $\iota_{\cal C}({\h x}) = 0$, for ${\h x} \in {\cal C}$, and $\iota_{\cal C}({\h x}) = +\infty$, otherwise, denotes the {\it indicator function} of the set ${\cal C}$. For a function $f:{\mathbb R}^n\rightarrow{\overline{\mathbb R}}$, we denote its {\it effective domain}  by $\dom f:=\{{\h x} \in {\mathbb R}^n: f({\h x})  < +\infty \}$ and say that it is {\it proper} if $\dom f \neq \emptyset$ and $f(\h x) > - \infty$ for all $x\in \R^n$. For $\overline{\h x} \in \R^n$ with $f(\overline{\h x}) \in \R$, the {\it Fr\'{e}chet subdifferential} of $f$ at $\overline{\h x}$ is defined as
\begin{equation*}
		{\widehat \partial} f({\overline{\h x}}):=\left\{{\h v} \in {\mathbb R}^n :{\liminf \limits_{\h x\to{\overline{\h x}}, \ {\h x}\neq {\overline{\h x}}}} \frac{f(\h x)-f(\overline{\h x})-\langle {\h v},{\h x}-{\overline{\h x}}\rangle}{\|{\h x}-\overline{\h x}\|}\ge 0\right\}.
\end{equation*}
Otherwise, we set ${\widehat \partial} f({\overline{\h x}}):= \emptyset$. For $\overline{\h x} \in \R^n$ with $f(\overline{\h x}) \in \R$, the {\it limiting subdifferential} of $f$ at $\overline{\h x}$ is defined as
\begin{equation*}
		\partial f({\overline{\h x}}):=\left\{{\h v} \in \Rbb^n \! : \! \exists\; ({\h x}^k) \rightarrow {\overline{\h x}}, \ f({\h x}^k)\rightarrow f(\overline{\h x}), ({\h v}^k) \rightarrow{\h v} \ \text{as} \ k \rightarrow +\infty, \
 {\h v}^k\in{\widehat \partial} f({\h x}^k) \right\}.
		\end{equation*}
Otherwise, we set $\partial f({\overline{\h x}}):= \emptyset$.
If $f$ is a convex function and $\varepsilon \geq 0$, for $\overline{\h x} \in \R^n$ with $f(\overline{\h x}) \in \R$ we denote by
\begin{equation}
\partial_{\varepsilon} f (\overline{\h x}) := \{{\h v} \in \Rbb^n : f({\h x}) \geq f(\overline{\h x}) + {\h v}^\top ({\h x} - \overline{\h x}) - \varepsilon \ \forall {{\h x} \in \R^n} \}
\end{equation}
the {\it $\varepsilon$-subdifferential} of $f$ at $\overline{\h x}$. Otherwise, we set $\partial_{\varepsilon} f (\overline{\h x}) := \emptyset$. For $\overline{\h x} \in \R^n$ with $f(\overline{\h x}) \in \R$, it holds
${\h v} \in \partial_{\varepsilon} f(\overline{\h x})$ if and only if $ f^*({\h v}) + f(\overline{{\h x}}) - \langle {\h v}, \overline{\h x} \rangle \leq \varepsilon$, where
$f^* : \Rbb^n \rightarrow {\overline{\mathbb R}}$, $f^*({\h v}) = \sup_{{\h x} \in \Rbb^n} \{\langle{\h v},{\h x} \rangle -f({\h x}) \}$, denotes the {\it (Fenchel) conjugate function} of $f$. The {\it convex subdifferential} of $f$ at $\overline{\h x}$ is defined by $\partial f (\overline{\h x}):=\partial_{0} f (\overline{\h x})$ and it coincides in this case with the Fr\'echet and the limiting subdifferential. For $\overline{\h x} \in \R^n$ with $f(\overline{\h x}) \in \R$, it holds ${\h v} \in \partial f(\overline{\h x})$ holds if and only if $ f^*({\h v}) + f(\overline{{\h x}}) = \langle {\h v}, \overline{\h x} \rangle$. The {\it domain} of the convex subdifferential is defined as $\dom \partial f:=\{{\h x} \in \R^n : \partial f({\h x}) \neq \emptyset\}$. Given a nonempty and convex set ${\cal C} \subseteq {\mathbb R}^n$, the set $N_{\mathcal{C}}({\h x}) := \partial \iota_{\mathcal{C}}(\h x)$ denotes the {\it normal cone} of ${\cal C}$ at ${\h x} \in {\mathbb R}^n$.

The following lemma, which can be found in \cite{BLT23}, will play a central role in establishing convergence to an approximate KKT point.

\begin{lemma}\label{epsappsol}
 Let $g:{{\mathbb R}^n}\rightarrow{\overline{\mathbb R}}$ be a proper, convex and lower semicontinuous function and ${\h w}\in{\text{\rm int}}({\text{\rm dom }}g)$.
Let $\varepsilon>0$ and ${\cal O}$ be a compact set such that $B({\h w},\varepsilon)\subseteq {\cal O} \subseteq {\text{\rm int}}({\text{\rm dom }}g)$ and $g$ is Lipschitz continuous on $\cal O$ with Lipschitz constant $\kappa$. Let  ${\widetilde{\h z}}\in{\text{\rm ri}}({\text{\rm dom}} \, g^*)$ be such that
 ${\text{\rm dist}}(\partial g^*({\widetilde {\h z}}),{\h w})\le \varepsilon$. Then it holds
 \begin{eqnarray}\label{Lem5:con} {\widetilde {\h z}}\in \partial_{\widehat{\varepsilon}} g({\h w}),\end{eqnarray}
 where ${\widehat\varepsilon}=2\kappa\varepsilon$.
 Furthermore, if $g$ is differentiable and $\nabla g$ is Lipschitz continuous with constant $\ell_{\nabla g}$, then one can set
 ${\widehat\varepsilon}=\frac{\ell_{\nabla g}}{2}\varepsilon^2$.
 \end{lemma}

Given a proper, convex and lower semicontinuous function  $f : {\mathbb R}^n \to \overline{\mathbb R}$, its {\it Moreau envelope} with parameter $\alpha >0$ is defined as (\cite[Definition 1.22]{RockWets})
\begin{eqnarray}\label{calG}
{\cal G}_{f,\alpha} : {\mathbb R}^n \to {\mathbb R}, \quad {\cal G}_{f,\alpha}({\h w})=\min_{\h x \in {\mathbb R}^n} \left \{f({\h x})+\frac{1}{2\alpha}\|{\h x}-{\h w}\|^2 \right \},\end{eqnarray}
 and its corresponding {\it proximal mapping}
 $$\prox_{f,\alpha} : {\mathbb R}^n \to {\mathbb R}^n, \quad \prox_{f,\alpha}({\h w}) = \argmin_{\h x \in {\mathbb R}^n} \Big\{ f({\h x}) + \frac{1}{2\alpha}\|{\h x}-{\h w}\|^2\Big\}.$$
 For every ${\h w} \in {\mathbb R}^n$ and every $\alpha >0$ it holds
 \begin{eqnarray*} \displaystyle{\prox_{f,\alpha}({\h w})+\alpha \prox_{f^*,1/\alpha}\left(\frac{{\h w}}{\alpha}\right)} = {\h w}. \end{eqnarray*}
 When $f({\h x})=\|{\h x}\|_1$, its proximal operator can be expressed by
\begin{eqnarray*}
\prox_{f,\alpha}({\h w}):=\text{\rm{Shrink}}({\h w},\alpha)= \max(|{\h w}|-\alpha,0)\odot{\text{sign}}({\h w}),\end{eqnarray*}
where $\odot$ denotes componentwise multiplication.

A proper function $f: \mathbb{R}^n \to \overline{\mathbb{R}}$ is said to be \textit{essentially strictly convex} if it is strictly convex on every nonempty convex subset of $\dom \partial f$. A proper and convex function $f: \mathbb{R}^n \to \overline{\mathbb{R}}$ is said to be \textit{essentially smooth} (see \cite{BC17}) if it satisfies the following three conditions:
\begin{itemize}
    \item[(a)] ${\rm int}(\dom f)  \neq \emptyset$,
    \item[(b)] $f$ is differentiable on ${\rm int}(\dom f)$,
    \item[(c)] For every sequence $(x^k)_{k \geq 0} \subseteq {\rm int}(\dom f)$ converging to a boundary point of $\dom f$, it holds
    \[
    \lim_{k \to +\infty} \|\nabla f(x_k)\| = +\infty.
    \]
\end{itemize}
For a proper, convex, and lower semicontinuous function $f: \mathbb{R}^n \to \overline{\mathbb{R}}$, it is known (see \cite[Corollary 18.12]{BC17}) that it is essentially strictly convex if and only if its conjugate $f^*$ is essentially smooth.

Given a linear operator $A: {\mathbb R}^n \rightarrow {\mathbb R}^m$, we denote by $A^*: {\mathbb R}^m \rightarrow {\mathbb R}^n$ its {\it adjoint operator}. We also denote by $\sigma_A$ the {\it maximal eigenvalue} of the matrix $A^*A$. It holds $\sigma_A= \|A\|^2$.

 For the differentiable function $\varphi:{\mathbb R}^n \rightarrow {\mathbb R}$ with $\ell_{\nabla \varphi}$-Lipschitz continuous gradient it holds (see, for instance, \cite{BST14})
\begin{equation}\label{eq:descent}
\varphi({\h u})\le \varphi({\h v}) + \langle {\h u}-{\h v},\nabla \varphi({\h v}) \rangle+\frac{\ell_{\nabla \varphi}}{2}\|{\h u}-{\h v}\|^2 \quad \forall\; {\h u},{\h v}\in{\mathbb R}^n.
\end{equation}

The following lemma will play a central role in the convergence analysis.  For its proof we refer to \cite[Lemma 5.31]{BC17}.

\begin{lemma}\label{ak}
Let \!\! $(a_k)_{k \geq 0}$ be a sequence of real numbers bounded from below, $(b_k)_{k \geq 0}$ a sequence of nonnegative real numbers, and $(\varepsilon_k)_{k \geq 0}$ a summable sequence of nonnegative real numbers such that
\[
 \quad a_{k+1} \leq  a_k - b_k + \varepsilon_k \quad \forall k \geq 0.
\]
Then $(a_k)_{k \geq 0}$ is convergent and $\sum_{k = 0}^{+\infty} b_k < +\infty$.
\end{lemma}

\section{Basic assumptions and stationary points of structured DC programs} \label{sec:Assumptions}

In this section, we set out the basic assumptions and review the concept of KKT and critical points for the structured DC program  (\ref{PForm}).

\begin{ass} \label{ass1} Throughout this paper, we assume that
\begin{itemize}
    \item[(a)]  
        ${A(S)\subseteq {\rm ri}({\dom g})}$;
    \item[{(b)}] $ \underline{F}:=\inf_{\h x\in S}\{ g(A{\h x})+\varphi({\h x})-f(K{\h x})\}>-\infty$.
\end{itemize}
\end{ass}

From Assumption \ref{ass1} (b) and the Fenchel-Moreau Theorem, we see that
\begin{align*} -\infty < \underline{F} = & \inf_{ \h x \in S}\{ g(A{\h x})+\varphi({\h x})-f(K{\h x})\}\\
= &  \inf_{{\h x} \in S,{\h y} \in \mathbb{R}^p}\sup_{\h z \in \mathbb{R}^s} \left \{\langle {\h z}, A{\h x}\rangle-g^*({\h z}) +\varphi({\h x})-\langle K{\h x}, {\h y} \rangle+f^*({\h y}) \right\}.\end{align*}

\begin{definition}\label{def:critical}  We call the triple $(\overline{\h x}, \overline{\h y},\overline{\h z}) \in\mathbb{R}^n\times \mathbb{R}^p \times \mathbb{R}^s$ a KKT point of  (\ref{PForm}) if
\begin{eqnarray}\label{critcal}
A{\overline{\h x}}\in \partial g^*({\overline{\h z}}), \quad
 K {\overline{\h x}}\in\partial f^*({\overline{\h y}}) \quad \mbox{and} \quad
K^* {\overline{\h y}} \in A^* {\overline{\h z}} + \nabla\varphi({\overline{\h x}}) +N_S(\overline{\h x}).
\end{eqnarray}
\end{definition}

We notice that ${\overline{\h x} \in \mathbb{R}^n}$ is a critical point of the optimization problem (\ref{PForm}), i.e.,
  \begin{eqnarray}\label{critical} \Big( A^* {\partial} g(A{\overline{\h x}}) + \nabla \varphi({\overline{\h x}}) + N_S(\overline{\h x}) \Big) \cap K^* \partial f(K{\overline{\h x}}) \neq \emptyset
  \end{eqnarray}
    if and only if there exists a pair $(\overline{\h y},\overline{\h z}) \in\mathbb{R}^p \times \mathbb{R}^s$ such that $(\overline{\h x}, \overline{\h y},\overline{\h z})$ is a KKT point of (\ref{PForm}).

Inspired by Definition \ref{def:critical}, we introduce the notion of an approximate KKT point of the structured DC program (\ref{PForm}).

\begin{definition}\label{appsol} Given $\varepsilon\ge 0$, we call a triple $(\overline{\h x}, \overline{\h y},\overline{\h z}) \in\mathbb{R}^n\times \mathbb{R}^p \times \mathbb{R}^s$ an $\varepsilon$-approximate KKT point of  (\ref{PForm}) if
\begin{eqnarray}\label{approx critcal}
A{\overline{\h x}}\in \partial_\varepsilon g^*({\overline{\h z}}), \quad
 K {\overline{\h x}}\in\partial f^*({\overline{\h y}}) \ \mbox{and} \
K^* {\overline{\h y}} \in A^* {\overline{\h z}} + \nabla\varphi({\overline{\h x}}) +N_S(\overline{\h x}).
\end{eqnarray}
\end{definition}

We notice that that ${\overline{\h x}} \in \mathbb{R}^n$ is an ${{\varepsilon}}$-approximate critical point of the optimization problem (\ref{PForm}), i.e.,
\begin{eqnarray*}
{\Big ( A^* {\partial}_{{\varepsilon}} g(A{\overline{\h x}}) + \nabla \varphi({\overline{\h x}}) + N_S(\overline{\h x})\Big) \cap K^* \partial f(K{\overline{\h x}} ) \neq \emptyset}\end{eqnarray*}
  if and only if there exists a pair $(\overline{\h y},\overline{\h z}) \in\mathbb{R}^p \times \mathbb{R}^s$ such that $(\overline{\h x}, \overline{\h y},\overline{\h z})$ is an ${{\varepsilon}}$-approximate KKT point of  (\ref{PForm}).

For $\varepsilon =0$, $\varepsilon$-approximate KKT points are nothing more than KKT points, and the same is true for ${{\varepsilon}}$-approximate critical points and critical points.

\section{Double proximal full splitting algorithm} \label{sec:DPFS}

\subsection{Conceptual algorithmic framework}

In this subsection, our aim is to formulate a first-order algorithm characterized by a ``single-loop" structure, grounded in the overarching concept of full splitting as established in previous  work \cite{SB19, BA20, BLT23}. This approach handles proximal evaluations for each nonsmooth component independently, avoiding proximal evaluations for nonsmooth compositions involving linear operators. Following these principles from \cite{BLT23}, we use the gradient for the smooth component $\varphi$, along with the evaluation of the convex functions and the linear operators involved in the compositions $g\circ A$ and $f\circ K$ separately. Consequently, we propose a so-called double proximal full splitting algorithm ({DPFS}), which is defined for all $k\ge 0$ as
\begin{subequations}
            \label{SplitI}
	\begin{numcases}{\hbox{\quad}}
  \h x^{k+1}\! = { \text{Proj}_{S} \left[ \!{\h x}^k-\frac{1}{\rho_k}\left( A^*{\h z}^k- K^*{\h y}^k+\nabla \varphi({\h x}^k)\right) \right]},\label{xsub} \\
	\label{ysub}{\h{y}}^{k+1}=\argmin_{{\h y} \in \mathbb{R}^p}\left[ f^*({\h y})-\langle K{\h x}^{k+1}, {\h y}\rangle+\frac{\beta_k}{2}\|\h y-{\h y}^k\|^2\right],\\[0.0cm]
\label{zsub}{\h{z}}^{k+1}=\argmin_{\h z \in \mathbb{R}^s}\left[ g^*({\h z})-\langle A{\h x}^{k+1}, {\h z}\rangle+\frac{\alpha_k}{2}\|\h z-\delta_k{\h z}^k\|^2\right].
	\end{numcases}
\end{subequations}
The parameter sequences $(\rho_k)_{k\ge 0},(\beta_k)_{k\ge 0},(\alpha_k)_{k\ge 0},(\delta_k)_{k\ge 0}$ are all positive and will be specified later.

\subsection{DPFS}

In this subsection, we will specify the parameter sequences used in the construction of the sequence $({\h x_{k}}, {\h y}_{k}, {\h z}_{k})_{k\geq 0}$ in (\ref{SplitI}), and establish the first steps of the convergence analysis. Second, we provide a general framework for that adaptive step size sequences that promote the descent and convergence of a discrete energy sequence. Thirdly, we show that the subsequential convergence of the sequence generated by DPFS to an exact KKT point cannot be guaranteed if the step sizes are vanishing.

The discrete energy sequence used in the convergence analysis is defined for all $k \geq 2$ by
\begin{align}\label{Lfun2}
\psi_{k} := & \ \Psi({\h x}^{k}, {\h y}^{k}, {\h z}^{k}) + \alpha \left(\frac{1}{2} +\frac{2\alpha}{(\alpha_{k}-\alpha)}\right)\|{\h z}^{k}-{\h z}^{k-1}\|^2 \notag \\
& \ - \left(\frac{1}{2}(\alpha_{k-2} - \alpha) - 2\alpha\frac{\alpha_{k-1} - \alpha_{k-2}}{\alpha_{k-1} - \alpha} \right) \norm{\h z^{k}}^2,
\end{align}
where $\alpha_{k}\geq \alpha>0$, and
\begin{eqnarray}\label{Phif}
\Psi({\h x},{\h y}, {\h z}) := \langle{\h z}, A{\h x} \rangle - g^*({\h z}) +\varphi({\h x}) + { \iota_S({\h x})} - \langle K{\h x},{\h y}\rangle +f^*({\h y}).
\end{eqnarray}

\begin{theorem}\label{PriTheo} Suppose that the parameter sequences $(\alpha_k)_{k\ge 0}$, $(\delta_k)_{k\ge 0}$, $(\beta_k)_{k\ge 0}$, and $(\rho_k)_{k\ge0}$ satisfy
$$\alpha_k \geq \alpha_{k+1} > \alpha >0, \quad \delta_k = \frac{\alpha}{\alpha_k}, \quad \beta_k > 0 \quad \mbox{and} \quad \rho_k > \frac{\ell_{\nabla \varphi}}{2} \quad \forall k \geq 0.$$
Let {$({\h x}^0, {\h y}^0,{\h z}^0)\in S \times \R^p \times \R^s$} and $({\h x}^k,{\h y}^k,{\h z}^k)_{k\ge 0}$ be the sequence generated by (\ref{SplitI}). We define
\begin{align} \label{eq:def_omega_k}
\omega_k  := \left[\frac{1}{2}(\alpha_{k-2} - \alpha_{k-1}) + 2\alpha\left(\frac{\alpha_{k-2} - \alpha_{k-1}}{\alpha_{k-1} - \alpha} - \frac{\alpha_{k-1} - \alpha_k}{\alpha_k - \alpha}\right)\right]\|{\h z}^{k}\|^2 \quad \forall k \geq 2.
\end{align}
Then, for all $k \ge 2$, it holds
\begin{equation}\label{descent}
\begin{aligned}
&\psi_{k+1} \le \psi_{k} - c_1^k\|{\h x}^k - {\h x}^{k+1}\|^2 - c_2^k\|{\h y}^k - {\h y}^{k+1}\|^2 - c_3^k\|{\h z}^{k+1} - {\h z}^{k}\|^2 + \omega_k,
\end{aligned}
\end{equation}
where $c_1^k$, $c_2^k$, $c_3^k$ are defined by
\begin{align*}
c_1^k & := \frac{1}{2}\left(2\rho_k - \ell_{\nabla \varphi} - 2\sigma_A\left(\frac{1}{2(\alpha_{k-1} - \alpha)} + \frac{2\alpha}{(\alpha_k - \alpha)^2}\right)\right),\\
c_2^k & := \beta_k, \\
c_3^k & := \alpha\left(1 - 2\alpha\frac{\alpha_k - \alpha_{k+1}}{(\alpha_{k+1} - \alpha)(\alpha_k - \alpha)}\right).
\end{align*}
\end{theorem}

\begin{proof} Let $k \geq 2$ fixed. Since, by (\ref{xsub}),
\begin{eqnarray*}
{\h x}^{k+1} = \argmin_{{\h x} \in \R^n} \left( \iota_S({\h x}) + \langle {\h z}^k, A{\h x} \rangle + \langle {\h x} - {\h x}^k, \nabla \varphi({\h x}^k) \rangle - \langle K{\h x}, {\h y}^k \rangle + \frac{\rho_k}{2}\|{\h x} - {\h x}^k\|^2 \right),
\end{eqnarray*}
using the strong convexity of the objective function in the above optimization problems, it yields
\begin{align*}
& \ \langle {\h z}^k, A{\h x}^{k+1} \rangle + \langle {\h x}^{k+1} - {\h x}^k, \nabla \varphi({\h x}^k) \rangle - \langle K{\h x}^{k+1}, {\h y}^{k} \rangle + \rho_k \|{\h x}^{k+1} - {\h x}^k\|^2 \\
\le & \ \langle {\h z}^k, A{\h x}^{k} \rangle - \langle K{\h x}^{k}, {\h y}^{k} \rangle.
\end{align*}
From \eqref{eq:descent} it follows that
\[
\varphi({\h x}^{k+1}) \le \varphi({\h x}^k) + \langle {\h x}^{k+1} - {\h x}^k, \nabla \varphi({\h x}^k) \rangle + \frac{\ell_{\nabla \varphi}}{2}\|{\h x}^{k+1} - {\h x}^k\|^2.
\]
Adding the above two statements, yields
\begin{align*} & \ \varphi({\h x}^{k+1}) + \langle{\h z}^k, A{\h x}^{k+1} \rangle  - \langle K{\h x}^{k+1},{\h y}^k\rangle \nn\\
\le & \
\varphi({\h x}^{k})+
\langle{\h z}^k, A{\h x}^{k} \rangle  - \langle K{\h x}^{k},{\h y}^k\rangle -\frac{{ 2\rho_k}-\ell_{\nabla \varphi}}{2}\|{\h x}^k-{\h x}^{k+1}\|^2, \end{align*}
which, by invoking the definition of $\Psi$ in (\ref{Phif}), is nothing but
\begin{eqnarray} \label{xsubdesr} \Psi({\h x}^{k+1},{\h y}^{k}, {\h z}^{k})\le \Psi({\h x}^{k},{\h y}^{k}, {\h z}^{k})-\frac{{2\rho_k}-\ell_{\nabla \varphi}}{2}\|{\h x}^k-{\h x}^{k+1}\|^2. \end{eqnarray}

On the other hand, from (\ref{ysub}) we get that
$$f^*({\h y}^{k+1})- \langle K{\h x}^{k+1},{\h y}^{k+1} \rangle+\frac{\beta_k}{2}\|{\h y}^{k+1}-{\h y}^k\|^2\le f^*({\h y}^{k})- \langle K{\h x}^{k+1},{\h y}^{k} \rangle,$$
which, by invoking again the definition of $\Psi$, is nothing but
\begin{eqnarray} \label{ysubdesr} \Psi({\h x}^{k+1},{\h y}^{k+1}, {\h z}^{k})\le \Psi({\h x}^{k+1},{\h y}^{k}, {\h z}^{k})-\frac{\beta_k}{2}\|{\h y}^k-{\h y}^{k+1}\|^2. \end{eqnarray}

Further, from (\ref{zsub}) we have that
$$ A{\h x}^{k}-\alpha_{k-1}({\h z}^k-\delta_{k-1}{\h z}^{k-1})\in \partial g^*({\h z}^k),$$
which leads to
\begin{align}\label{gfun}
-g^*({\h z}^{k+1}) + \langle {\h z}^{k+1}, A{\h x}^{k+1}\rangle \le & \ -g^*({\h z}^k)+ \langle {\h z}^{k}, A{\h x}^{k+1}\rangle\nn\\
& + \langle{\h z}^{k+1}-{\h z}^k,A{\h x}^{k+1} -  \!(A{\h x}^{k}-\alpha_{k-1}({\h z}^k-\delta_{k-1}{\h z}^{k-1}))\rangle.
\end{align}
We can regroup
\begin{align*}
   & \ \langle{\h z}^{k+1}-{\h z}^k,A{\h x}^{k+1} - (A{\h x}^{k}-\alpha_{k-1}({\h z}^k-\delta_{k-1}{\h z}^{k-1}))\rangle \\
    = & \ \skal{{\h z}^{k+1}-{\h z}^k,(A{\h x}^{k+1} - A{\h x}^{k}) + (\alpha_{k-1}-\alpha){\h z}^k + \alpha ({\h z}^k-{\h z}^{k-1})},
\end{align*}
and estimate the individual parts as follows
\begin{align*}
    \skal{{\h z}^{k+1}-{\h z}^k,A{\h x}^{k+1} - A{\h x}^{k}} & \leq \frac{\sigma_A}{2(\alpha_{k-1}-\alpha)} \norm{{\h x}^k-{\h x}^{k+1}}^2 + \frac{1}{2}(\alpha_{k-1}-\alpha)\|{\h z}^k-{\h z}^{k+1}\|^2, \\
    (\alpha_{k-1}-\alpha)\skal{{\h z}^{k+1}-{\h z}^k,{\h z}^k} & = \frac{1}{2}(\alpha_{k-1}-\alpha)(\|{\h z}^{k+1}\|^2-\|{\h z}^k\|^2 -\|{\h z}^{k}-{\h z}^{k+1}\|^2), \\
   \alpha\skal{{\h z}^{k+1}-{\h z}^k,  {\h z}^k-{\h z}^{k-1}} & \leq \frac{\alpha}{2}\left( \|{\h z}^k-{\h z}^{k+1} \|^2+\|{\h z}^k-{\h z}^{k-1} \|^2\right) .
\end{align*}
Plugging these inequalities into \eqref{gfun}, and using again the definition of $\Psi$, it yields
\begin{align}\label{zsubdes}
& \ \Psi({\h x}^{k+1},{\h y}^{k+1}, {\h z}^{k+1})  -\frac{1}{2}(\alpha_k-\alpha)\|{\h z}^{k+1}\|^2 \nn \\
\leq  & \ \Psi({\h x}^{k+1},{\h y}^{k+1}, {\h z}^{k}) -\frac{1}{2}(\alpha_{k-1}-\alpha)\|{\h z}^{k}\|^2\nn\\
& +\frac{\alpha}{2}(\|{\h z}^k-{\h z}^{k+1}\|^2+\|{\h z}^{k-1}-{\h z}^{k}\|^2) +\frac{\sigma_A}{2(\alpha_{k-1}-\alpha)}\|{\h x}^k-{\h x}^{k+1}\|^2 \nn\\
&+\frac{1}{2}(\alpha_{k-1}-\alpha_k)\left(\|{\h z}^{k+1}\|^2-\|{\h z}^k\|^2\right) +\frac{1}{2}(\alpha_{k-1}-\alpha_k)\|{\h z}^k\|^2. \end{align}
On the other hand, we have
 \begin{align*}
& A{\h x}^{k+1}-\alpha_k{\h z}^{k+1}+\alpha{\h z}^k\in\partial g^*({\h z}^{k+1}),\\
& A{\h x}^{k}-\alpha_{k-1}{\h z}^{k}+\alpha{\h z}^{k-1}\in\partial g^*({\h z}^{k}),
 \end{align*}
and the monotonicity of $\partial g^*$ gives
\begin{align}
    0 \leq & \skal{{\h z}^k-{\h z}^{k+1},A{\h x}^{k} -A{\h x}^{k+1} + \alpha_k  {\h z}^{k+1} - \alpha_{k-1}{\h z}^k + \alpha({\h z}^{k-1} - {\h z}^k)  } \notag  \\
    = &  \skal{{\h z}^k-{\h z}^{k+1}, A{\h x}^{k} - \! A{\h x}^{k+1} + \!(\alpha_k - \alpha_{k-1})  {\h z}^{k+1} \! - \! \alpha_{k-1}({\h z}^k-{\h z}^{k+1}) + \alpha({\h z}^{k-1} - {\h z}^k)  } \notag . \\ &\label{eq:max_mon_g*}
\end{align}
We estimate some of the above scalar products by using the Cauchy-Schwarz inequality
\begin{align*}
    \langle{{\h z}^k-{\h z}^{k+1},A{\h x}^{k} -A{\h x}^{k+1}}\rangle &\leq \frac{\sigma_A}{2(\alpha_{k}-\alpha)} \norm{{\h x}^k-{\h x}^{k+1}}^2 + \frac{1}{2}(\alpha_{k}-\alpha)\|{\h z}^k-{\h z}^{k+1}\|^2, \\
    \alpha \skal{{\h z}^k-{\h z}^{k+1},{{\h z}^{k-1}-{\h z}^{k}}} & \leq \frac{\alpha}{2} \left( \norm{{\h z}^k-{\h z}^{k+1}}^2 + \norm{{\h z}^k-{\h z}^{k-1}}^2 \right),
\end{align*}
and rewrite the others as
\begin{align*}
    (\alpha_{k}- \alpha_{k-1})\skal{{\h z}^k-{\h z}^{k+1},{\h z}^{k+1}} &= \frac{1}{2}(\alpha_{k}- \alpha_{k-1})  \left(\|{\h z}^k\|^2-\|{\h z}^{k+1}\|^2- \norm{{\h z}^{k}-{\h z}^{k+1}}^2\right),  \\
    -\alpha_{k-1}\skal{{\h z}^k-{\h z}^{k+1},{\h z}^k-{\h z}^{k+1}} &= - \alpha_{k-1} \norm{{\h z}^k-{\h z}^{k+1}}^2.
\end{align*}
Plugging the above into \eqref{eq:max_mon_g*} gives
\begin{align*}
    0 \leq &-\frac{1}{2}(\alpha_k-\alpha)\|{\h z}^k-{\h z}^{k+1}\|^2+\frac{\sigma_A}{2(\alpha_k-\alpha)}\|{\h x}^k-{\h x}^{k+1}\|^2\\&+\frac{1}{2}(\alpha_k-\alpha_{k-1})\left[\|{\h z}^k\|^2-\|{\h z}^{k+1}\|^2\right] -\frac{\alpha}{2} \left(\|{\h z}^k-{\h z}^{k+1}\|^2-\|{\h z}^k-{\h z}^{k-1}\|^2\right),
\end{align*}
and then after multiplication by $\frac{4\alpha}{\alpha_k-\alpha}$
\begin{align*} \label{DCajc}
& \frac{2\alpha^2}{\alpha_{k+1}-\alpha}\|{\h z}^k-{\h z}^{k+1}\|^2 \\
\le & \ \frac{2\alpha^2}{\alpha_{k}-\alpha}\|{\h z}^k-{\h z}^{k-1}\|^2+\frac{2\alpha\sigma_A}{(\alpha_k-\alpha)^2}\|{\h x}^k-{\h x}^{k+1}\|^2\\
& -\left(2\alpha-2\alpha^2\frac{\alpha_k-\alpha_{k+1}}{(\alpha_{k+1}-\alpha)(\alpha_k-\alpha)}\right) \|{\h z}^k-{\h z}^{k+1}\|^2+2\alpha \frac{\alpha_k-\alpha_{k-1}}{\alpha_k-\alpha}(\|{\h z}^k\|^2 -\|{\h z}^{k+1}\|^2 ).
\end{align*}
Adding the above inequality to (\ref{zsubdes}) yields
\begin{align*}
& \Psi({\h x}^{k+1},{\h y}^{k+1},{\h z}^{k+1})-\frac{1}{2}(\alpha_k-\alpha)\|{\h z}^{k+1}\|^2+\left(\frac{2\alpha^2}{\alpha_{k+1}-\alpha}+\frac{\alpha}{2}\right)\|{\h z}^k-{\h z}^{k+1}\|^2\nn\\
\le & \ \Psi({\h x}^{k+1},{\h y}^{k+1},{\h z}^{k})-\frac{1}{2}(\alpha_{k-1}-\alpha)\|{\h z}^{k}\|^2+\left(\frac{2\alpha^2}{\alpha_{k}-\alpha}+\frac{\alpha}{2}\right)\|{\h z}^{k-1}-{\h z}^{k}\|^2\nn\\
& +\sigma_A\left(\frac{1}{2(\alpha_{k-1}-\alpha)}+\frac{2\alpha}{(\alpha_k-\alpha)^2}\right)\|{\h x}^k-{\h x}^{k+1}\|^2\nn\\
& -\alpha\left(1-2\alpha\frac{\alpha_k-\alpha_{k+1}}{(\alpha_{k+1}-\alpha)(\alpha_k-\alpha)}\right)\|{\h z}^k-{\h z}^{k+1}\|^2\nn\\
& +\underbrace{\left[2\alpha\frac{\alpha_k-\alpha_{k-1}}{\alpha_k-\alpha}-\frac{1}{2}(\alpha_{k-1}-\alpha_k)\right](\|{\h z}^k\|^2-\|{\h z}^{k+1}\|^2)+\frac{1}{2}(\alpha_{k-1}-\alpha_k)\|{\h z}^k\|^2}_\triangledown.\nn\\
\end{align*}
We regroup the terms in $\triangledown$ as follows
\begin{align*}\triangledown = & \ \ \ \left(\frac{1}{2}(\alpha_{k-1}-\alpha_k)-2\alpha \frac{\alpha_k-\alpha_{k-1}}{\alpha_k-\alpha}\right) \|{\h z}^{k+1}\|^2\\
& -\left(\frac{1}{2}(\alpha_{k-2}-\alpha_{k-1})-2\alpha \frac{\alpha_{k-1}-\alpha_{k-2}}{\alpha_{k-1}-\alpha}\right) \|{\h z}^{k}\|^2 \\
& +\left(\frac{1}{2}(\alpha_{k-2}-\alpha_{k-1}) +2\alpha \left (\frac{\alpha_{k-2}-\alpha_{k-1}}{\alpha_{k-1}-\alpha}-\frac{\alpha_{k-1}-\alpha_k}{\alpha_k-\alpha} \right )\right)\|{\h z}^{k}\|^2.\end{align*}
We recall the definition of the discrete energy function given in (\ref{Lfun2})
\begin{align*}
    \psi_{k+1} := & \ \Psi({\h x_{k+1}}, {\h y}_{k+1}, {\h z}_{k+1}) + \alpha \left(\frac{1}{2} +\frac{2\alpha}{(\alpha_{k+1}-\alpha)}\right)\|{\h z}^{k+1}-{\h z}^k\|^2 \\
    & - \left(\frac{1}{2}(\alpha_{k-1} - \alpha) - 2\alpha\frac{\alpha_{k} - \alpha_{k-1}}{\alpha_{k} - \alpha} \right) \norm{\h z^{k+1}}^2
\end{align*}
and, by combining the previous inequality with (\ref{xsubdesr}) and (\ref{ysubdesr}), conclude
 \begin{eqnarray} \psi_{k+1}&&\le \psi_{k}
-\frac{1}{2}\left({2\rho_k}-\ell_{\nabla \varphi}-2\sigma_A\left(\frac{1}{2(\alpha_{k-1}-\alpha)}+\frac{2\alpha}{(\alpha_k-\alpha)^2}\right)\right)\|{\h x}^k-{\h x}^{k+1}\|^2\nn\\
&&\quad -\frac{\beta_k}{2}\|{\h y}^k-{\h y}^{k+1}\|^2-\alpha\left(1-2\alpha\frac{\alpha_k-\alpha_{k+1}}{(\alpha_{k+1}-\alpha)(\alpha_k-\alpha)}\right)\|{\h z}^k-{\h z}^{k+1}\|^2\nn\\
&&\quad +\left(\frac{1}{2}(\alpha_{k-2}-\alpha_{k-1})+2\alpha\left(\frac{\alpha_{k-2}-\alpha_{k-1}}{\alpha_{k-1}-\alpha}-\frac{\alpha_{k-1}-\alpha_k}{\alpha_k-\alpha}\right)\right)\|{\h z}^{k}\|^2. \notag  \end{eqnarray}
This is nothing but (\ref{descent}).
\end{proof}

\subsubsection{Construction of dynamical step sizes to foster convergence}

In the following, we outline a possible regime for the step size sequences arising in {DPFS}, which induces desirable convergence properties for the discrete energy sequence and consequently for the sequence of generated iterates.

\begin{theorem}\label{subsequentialtoStat} Let Assumption \ref{ass1} be true, and $\gamma>0$, $\frac{\gamma}{2} > \alpha > 0$ and $\mu >0$. Suppose that the parameter sequences $(\alpha_k)_{k\ge 0}$, $(\beta_k)_{k\ge 0}$, $(\delta_k)_{k\ge 0}$ and $({\rho_k})_{k\ge 0}$ in {DPFS} satisfy the following properties:
\begin{align}\label{scalar}
&\lim_{k\to +\infty} \alpha_{k} = \alpha \ \mbox{and} \ \alpha + \gamma \geq \alpha_0 \geq \alpha_k \geq \alpha_{k+1} > \alpha \quad \forall k \geq 0,
\end{align}
with
\begin{align}\label{scalar2}
    &\frac{\alpha_k - \alpha_{k+1}}{(\alpha_{k+1} - \alpha)(\alpha_k - \alpha)} \leq \frac{1}{\gamma} \ \mbox{and} \ (\alpha_{k}-\alpha)(\alpha_{k+2}-\alpha) \geq (\alpha_{k+1} - \alpha)^2 \quad \forall k \geq 0,
\end{align}
\begin{align} \label{betacond} +\infty > {\overline\beta} \geq \beta_k \geq {\underline\beta} > 0, \quad \forall k \geq 0,\end{align}
\begin{align} \label{deltacond} \delta_k := \frac{\alpha}{\alpha_k} \quad \forall k \geq 0,\end{align}
and
\begin{align} \label{rhocond} \rho_k:={\frac{\ell_{\nabla \varphi}}{2}+\frac{3\sigma_A\gamma}{2(\alpha_k-\alpha)^2}+\mu} \quad \forall k \geq 0.\end{align}
Let $({\h x}^k,{\h y}^k,{\h z}^k)_{k\ge 0}$ be the sequence generated by (\ref{SplitI}). If $({\h z}^k)_{k \geq 0}$ is bounded, then the following statements are true:
\begin{itemize}
\item[(i)] The sequence of $(\psi_{k})_{k\ge 2}$ is bounded from below.
\item[(ii)] The sequence of $(\psi_{k})_{k\ge 2}$ is convergent,
i.e.,
$$ \lim\limits_{k\to+\infty} \psi_{k} \in \R.$$
\item[(iii)] The sequences of discrete velocities vanish, i.e.,
$$ \mathop{\lim}\limits_{k \to+\infty} \|{\h x}^k-{\h x}^{k+1}\|=0,\quad \mathop{\lim}\limits_{k \to+\infty} \|{\h y}^k-{\h y}^{k+1}\|=0,\quad \mathop{\lim}\limits_{k \to+\infty} \|{\h z}^k-{\h z}^{k+1}\|=0 .$$
\end{itemize}
\end{theorem}

\begin{proof}
(i) Let $\ell \geq 0$ be such that $\|{\h z}^k\|\le \ell$ for all $k\geq 0$.  For all $k \geq 2$ it holds
\begin{align*}
    \frac{1}{2}(\alpha_{k-2}-\alpha)+2\alpha\frac{\alpha_{k-2}-\alpha_{k-1}}{\alpha_{k-1}-\alpha} = & \ (\alpha_{k-2}-\alpha) \left(\frac{1}{2}+2\alpha\frac{\alpha_{k-2}-\alpha_{k-1}}{(\alpha_{k-1}-\alpha)(\alpha_{k-2}-\alpha)} \right) \\
    \leq & \ \gamma \left( \frac{1}{2} + \frac{2\alpha}{\gamma}\right) \leq \frac{3}{2}\gamma,
\end{align*}
consequently,
$$ -\left(\frac{1}{2}(\alpha_{k-2}-\alpha)+2\alpha\frac{\alpha_{k-2}-\alpha_{k-1}}{\alpha_{k-1}-\alpha}\right)\norm{\h z^{k} }^2 \ge - \frac{3}{2}\gamma\ell^2.$$

By \eqref{zsub}, we have for all $k \geq 1$ that
\begin{align*}
g^*({\h z}^{k+1}) -\langle A{\h x}^{k+1},{\h z}^{k+1} \rangle &\le  g^*( {\h z}^1)- \langle A{\h x}^{k+1},{\h z}^{1} \rangle+\frac{\alpha_k}{2}\|{\h z}^1 - \delta_k {\h z}^k\|^2\nn\\
&\leq g^*( {\h z}^1)+ \norm{A{\h x}^{k+1}} \norm{{\h z}^{1}} +\alpha_0 \norm{{\h z}^1}^2 + \alpha_0 \delta_k^2 \norm{\h z^{k}}^2 \notag \\
&\leq g^*( {\h z}^1)+ \ell \norm{A{\h x}^{k+1}}  +2\alpha_0 \ell^2.
\end{align*}
We define ${\cal C}:=g^*( {\h z}^1)+2\alpha_0\ell^2+\sup_{{\h x}\in S}\{ g(A\h x) + \ell \norm{A{\h x}} \}$, which is well-defined since $S$ is a compact set and $A(S) \subseteq \mathrm{ri}(\dom g)$. Therefore we have that  for
all $k\geq 1$
\begin{align} \label{eq:F_is_bounded2}
\Psi({\h x}^{k+1},{\h y}^{k+1}, {\h z}^{k+1})&\ge-g^*( {\h z}^1)-2\alpha_0\ell^2 - \ell \norm{A {\h x}^{k+1}} -g(A{\h x}^{k+1}) \notag \\& \quad + g(A{\h x}^{k+1})+\varphi({\h x}^{k+1})-\langle K{\h x}^{k+1},{\h y}^{k+1} \rangle+ f^*({\h y}^{k+1})\notag \\
&\ge F({\h x}^{k+1})-{\cal C}\ge \underline{F}-{\cal C}>-\infty.
\end{align}
Therefore $(\psi_{k+1})_{k \geq 1}$ is bounded from below and assertion  (i) is fulfilled.

(ii) For all $k \geq 2$ it holds
\begin{align}\label{alphalowb}
\frac{2\alpha}{(\alpha_k-\alpha)^2}+ \frac{1}{2(\alpha_{k-1}-\alpha)} \le & \ \frac{2\alpha}{(\alpha_k-\alpha)^2} +\frac{1}{2(\alpha_k-\alpha)} \nn\\
\le & \ \frac{\gamma}{(\alpha_k-\alpha)^2} +\frac{\gamma}{2(\alpha_0-\alpha)(\alpha_k-\alpha)} < \frac{3\gamma}{2(\alpha_k-\alpha)^2},\end{align}
from where we deduce that
$$c_1^k = \frac{1}{2}\left(2\rho_k - \ell_{\nabla \varphi} - 2\sigma_A\left(\frac{1}{2(\alpha_{k-1} - \alpha)} + \frac{2\alpha}{(\alpha_k - \alpha)^2}\right)\right) > \mu>0.$$
Furthermore, for all $k \geq 2$, $c_2^k = \beta_k > 0$ and
$$c_3^k = \alpha\left(1 - 2\alpha\frac{\alpha_k - \alpha_{k+1}}{(\alpha_{k+1} - \alpha)(\alpha_k - \alpha)}\right) \ge \alpha\left(1-\frac{2\alpha}{\gamma}\right)> 0.$$

By the given assumptions on the sequence $(\alpha_k)_{k \geq 0}$, we have for all $k \geq 2$
\begin{align*}
    \left(\frac{\alpha_{k-2} - \alpha_{k-1}}{\alpha_{k-1} - \alpha}\right)
\;-\;
\left(\frac{\alpha_{k-1} - \alpha_{k}}{\alpha_{k} - \alpha}\right)
=& \;\frac{\alpha_{k-2} - \alpha}{\alpha_{k-1} - \alpha}
\;-\;
\frac{\alpha_{k-1} - \alpha}{\alpha_{k} - \alpha} \geq 0,
\end{align*}
consequently,
\begin{align*}
   0\leq  \omega_k \leq \left[\frac{1}{2}(\alpha_{k-2}-\alpha_{k-1}) + 2\alpha \left(\frac{\alpha_{k-2} - \alpha_{k-1}}{\alpha_{k-1} - \alpha}- \frac{\alpha_{k-1} - \alpha_{k}}{\alpha_{k} - \alpha}\right) \right] \ell^2.
\end{align*}
Summing these inequalities over $k=2\dots N$, we obtain
\begin{align*}
    0 \leq \sum_{k=2}^N \omega_k & < \left[\frac{1}{2}\alpha_{0} + 2\alpha \left(\frac{\alpha_{0} - \alpha_{1}}{\alpha_{1} - \alpha}\right) \right]\ell^2.
\end{align*}
Letting $N \to +\infty$, we derive that $\sum_{k=2}^{+\infty} \omega_k < + \infty $ and $\mathop{\lim}\limits_{k \to +\infty}\omega_k =0$.
Now, invoking equation \eqref{descent} and noting that the sequence $(\psi_k)_{k\geq 2}$ is bounded from below, by Lemma \ref{ak}, it follows that $\lim\limits_{k\to+\infty} \psi_{k} \in \R$ and
 \begin{eqnarray*} \sum_{k=2}^{+\infty} \Big(c_1^k\|{\h x}^k-{\h x}^{k+1}\|^2+c_2^k\|{\h y}^k-{\h y}^{k+1}\|^2
 +c_3^k\|{\h z}^{k+1}-{\h z}^{k}\|^2 \Big) <+\infty.\end{eqnarray*}

(iii) Since $$\varsigma :=\inf_{k \geq 0}\min(c_1^k, c_2^k, c_3^k)\ge\min(\mu,\alpha(1-2\alpha/\gamma),\inf_{k \geq 0} \beta_k) >0,$$
it yields
\begin{eqnarray*}\sum_{k=0}^{+\infty} \|{\h x}^k-{\h x}^{k+1}\|^2<+\infty,\;
                  \sum_{k=0}^{+\infty} \|{\h y}^k-{\h y}^{k+1}\|^2<+\infty,\;
                 \sum_{k=0}^{+\infty} \|{\h z}^k-{\h z}^{k+1}\|^2<+\infty. \end{eqnarray*}
                 Assertion (iii) follows directly.
\end{proof}
\begin{remark}[Explicit step size regime]
\label{remk4.3}
To simplify the arguments, let us choose $\alpha_0 := \alpha + \gamma $. Focusing on the assumptions stated in \eqref{scalar2} for the sequence $(\alpha_k)_{k \geq 0}$, we observe that for all $k \geq 0$
\begin{align*}
  0 \leq   \frac{1}{\alpha_{k+1}-\alpha} - \frac{1}{\alpha_{k}- \alpha }=\frac{\alpha_k - \alpha_{k+1}}{(\alpha_{k+1} - \alpha)(\alpha_k - \alpha)} \leq \frac{1}{\gamma}.
\end{align*}
Applying a telescoping sum over the indices $k=0,\dots,K$, we obtain
\begin{align*}
    \frac{1}{\alpha_{K+1}-\alpha} \leq \frac{K+1}{\gamma} + \frac{1}{\alpha_0 - \alpha} = \frac{K+2}{\gamma},
\end{align*}
for $K \geq 0$ and see that $\alpha_{K+1}\geq \alpha + \frac{\gamma}{K+2}$ has to hold. Moreover, the condition $(\alpha_{k}-\alpha)(\alpha_{k+2}-\alpha) \geq (\alpha_{k+1}-\alpha)^2$ ensures that the sequence $(\alpha_{k}-\alpha)_{k\geq0}$ exhibits log-convexity. This property holds, for instance, for sequences of the form $(\frac{c}{(k+1)^p})_{k\geq0}$ with $p>0$. Thus it is easy to see that the conditions in \eqref{scalar} and \eqref{scalar2} are satisfied for sequences of the form \begin{align}
       \alpha_k = \alpha + \frac{\gamma}{(k+1)^{r}} \label{propose:stepsize}
    \end{align}
with $r\in(0,1]$ for all $k\geq 0$.

\end{remark}

Although the discrete velocities vanish, as seen  in Theorem \ref{subsequentialtoStat}  (iii), this does yet not necessarily imply that the sequence generated by DPFS approaches the set of KKT points of \eqref{PForm}. In the following we give an counterexample which illustrates that {\it may not be possible to achieve even subsequential convergence to an KKT point due to the decreasing step size sequence $(\frac{1}{\rho_k})_{k\geq 1 }$}.

\subsubsection{Examining subsequential convergence to KKT points}\label{subseqnot}

We see in \eqref{rhocond} that the step size $\frac{1}{\rho_k}$ in the update \eqref{xsub} of DFPS tends to zero as iterations increase.  This  can be expected to affect the convergence properties of the algorithm. In the following, we present an example that serves to demonstrate that, indeed, if the parameters of $\rho_k$, $\beta_k$, $\alpha_k$, and $\delta_k$
are chosen according to the requirements of Theorem \ref{subsequentialtoStat}, the subsequential convergence of DPFS to an exact KKT point may not be achieved.

\begin{example}\label{ex1}
Consider the optimization problem (\ref{PForm}) with $n=p=s=2$, ${S}:=\{{\h x} \in \R^2:\|{\h x}\|_2\le {200} \}$, the functions $g({\h x})=\|{\h x}\|_1$, $\varphi({\h x})=-\frac{1}{2}\|{\h x}\|_2^2$, and $f({\h x})={\h c}^\top {\h x}$, where ${\h c}=(2,2)^\top$, and $A=K=I$, the identity operator on $\R^2$. Then $g^*=\iota_{{\cal B}_{\infty}}$, where ${\cal B}_{\infty}$ denotes the unit ball in the sup-norm in $\R^2$, and $f^*=\iota_{\{\h c\}}$.

In accordance to Theorem \ref{subsequentialtoStat}, we choose $\gamma = 3$, $\alpha = 1$, $\mu=\frac{1}{2}$,  $\alpha_{k} = \alpha + \frac{\gamma}{k+1}$,  and $${\rho_{k} = \frac{1}{2} + \frac{3}{2\gamma}(k+1)^2} + \mu = 1 + \frac{1}{2}(k+1)^2 \ \forall k \geq 0.$$ The iterative scheme (\ref{SplitI}) is for all $k \geq 0$ as follows
\begin{eqnarray*}
\left\{\begin{array}{l}
{\h x}^{k+1}={\text{\rm Proj}}_{S}\left[(1+\frac{1}{\rho_k}){\h x}^k-\frac{1}{\rho_k}{\h z}^k+\frac{1}{\rho_k}{\h c}\right],\\
{\h y}^{k+1}={\h c},\\
{\h z}^{k+1}= \proj_{{\cal B}_{\infty}}\left(\frac{{\h z}^k+ {\h x}^{k+1}}{\alpha_k}\right).
\end{array}\right.\end{eqnarray*}
For the starting point we assume that ${\h x}^0=({ 0,0})^\top$ and $\|{\h z}^0\|_{\infty}\le 1$.
We are going to show that $(\h x^k)_{k \geq 0}$ remains in the interior of $S$ and that none of the accumulation points of $(\h x^k, \h y^k, \h z^k)_{k \geq 0}$ is a KKT point of \eqref{ex1}.

We denote ${\h u}^{k} := (1+\frac{1}{\rho_k}){\h x}^k-\frac{1}{\rho_k}{\h z}^k+\frac{1}{\rho_k}{\h c} $ for all $k \geq 0$ and prove by induction that $\norm{{\h u}^k}_2 < 150$ for all $k \geq 0$.
By our choice, $\norm{{\h u}^0}_2 < 150$ holds. Assume that for $K\in \mathbb{N}$,  $ \norm{{\h u}^k}_2 < 150$ holds for all $0 \leq k\leq K$. This means that ${\h x}^{k+1}={\h u}^k$ for all $0 \leq k\leq K$.
Since $\|\h z^k\|_{\infty}\le 1$, we have that $1\le({\h c}-{\h z}^k)_i\le 3$ for $i\in \{1,2\}$ and all $0 \leq k \leq K$. Then, for all $0\leq  k\leq K$ and $i\in\{1,2\}$, we have
\begin{align} \label{eq:uk}
    \left(1+\frac{1}{\rho_k}\right)({\h x}^k )_i+\frac{1}{\rho_k}\le ({\h x}^{k+1})_i= ({\h u}^{k})_i\le \left(1+\frac{1}{\rho_k}\right)({\h x}^k)_i+\frac{3}{\rho_k},
\end{align}
therefore
\begin{align*}
    ({\h x}^{k+1})_i +3 \le \left(1+\frac{1}{\rho_k}\right)(({\h x}^k)_i+3) \le e^{\frac{1}{\rho_k} } (({\h x}^k)_i+3).
\end{align*}
Thus, for $i\in \{1,2\}$ it holds
\begin{align*}
    ({\h u}^{K+1})_i\leq \left(1+\frac{1}{\rho_{K+1}}\right)({\h x}^{K+1})_i+\frac{3}{\rho_{K+1}} &\leq
    e^{\sum_{k=0}^{K+1} \frac{1}{\rho_k}} (({\h x}^0)_i+3) - 3 \\
    &= 3e^{\sum_{k=0}^{K+1} \frac{1}{\rho_k}} - 3.
\end{align*}
Since $\sum_{k=0}^{K+1} \frac{1}{\rho_k} =  \sum_{k=0}^{K+1} \frac{2}{2 + (k+1)^2} < 1+\frac{\sqrt{2}\pi}{2}$, where we used an integral upper bound, we have that $3e^{\sum_{k=0}^{K+1} \frac{1}{\rho_k}} - 3 < 3e^{1 + \frac{\sqrt{2}\pi}{2}} -3 < 75 $. Thus $\norm{{\h u}^{K+1}}_2\leq 75 \sqrt{2}<150$, which completes the induction.

This shows that $ (\h u^k)_{k\ge 0} \subseteq \mathrm{int } (S)$ and therefore ${\h x}^{k+1} = (1+\frac{1}{\rho_k}){\h x}^k-\frac{1}{\rho_k}{\h z}^k+\frac{1}{\rho_k}{\h c}$ for all $k\geq 0$.  Since the sequence $(\h x^k, \h y^k, \h z^k)_{k\ge 0}$ is bounded, it has accumulation points. Let $(\overline{\h x},\overline{\h y},\overline{\h z})$ be such an accumulation point. Then $ \norm{\overline{\h x}}_2 \leq 150 $, $\overline{\h y}={\h c}$ and $\|\overline{\h z}\|_{\infty}\le 1$.

By Theorem \ref{subsequentialtoStat}, we have that $\overline{\h z} = \mathrm{Proj}_{\mathcal{B}_\infty}(\overline{\h z} + \overline{\h x})$, therefore $\overline{\h x}\in \partial \iota_{\mathcal{B}_\infty} (\overline{\h z})$. Furthermore, ${\overline{\h y}} = {\h c}$, so $\partial f^* (\overline{\h y})= \partial f^* ({\h c}) = \R^2$. By \eqref{eq:uk} we have that $({\h x}^k)_i > ({\h x}^0)_i=0$ for all $k\geq 1$, thus, for $i \in \{1,2\}$
\begin{align*}
    \Big(\overline{\h z} + \nabla \varphi( \overline{\h x}) + N_S(\overline{\h x})\Big)_i = (\overline{\h z} - \overline{\h x})_i \leq 1.
\end{align*}
We see that the third condition is \eqref{critcal} is not fulfilled. This shows that $(\overline{\h x}, \overline{\h y}, \overline{\h z})$ is not a KKT Point.
\end{example}
Example \ref{ex1} illustrates that, while the descent of the potential function is preserved, subsequential convergence of the sequence generated by DPFS to an exact KKT point is not guaranteed. In the next section we propose a criterion that allows the step sizes to remain constant after a finite number of iterations.  However, this comes at the cost of achieving only a \(\varepsilon\)-approximate KKT point, as defined in Definition \ref{appsol}.

\subsection{Adaptive DPFS}\label{sec:Adaptive_DPFS}

	\begin{algorithm}
	\caption{Adaptive DPFS}
	\label{AdaptiveDPFS}
	\begin{algorithmic}[1]
		\Statex{\textbf {given:} $\ell_{\nabla \varphi}>0$, $\sigma_A>0$.}
		\Statex{ \textbf{choose:} $\beta_0>0$,  $\rho_0>0$,   $0<\alpha<\frac{\gamma}{2}$, $\alpha_0 = \alpha + \gamma$,
			$\delta_0 = \frac{\alpha}{\alpha_0}$, $\mu >0$ and $\varepsilon>0$, and a starting point $({\h x}^0, {\h y}^0, {\h z}^0)$ .}
		\For{$k = 0, 1, 2, ...$}
			\State{${\h x}^{k+1} := \proj_S\left({\h x}^k + \frac{1}{\rho_k}\left( -\nabla \varphi({\h x}^k) + K^* {\h y}^k - A^* {\h z}^k\right)\right)$,}
			\State{${\h y}^{k+1} := \prox_{f^*,1/\beta_k}({\h y}^k + \frac{1}{\beta_k}K{\h x}^{k+1})$,}
			\State{${\h z}^{k+1} := \prox_{g^*,1/\alpha_k}(\delta_k{\h z}^k + \frac{1}{\alpha_k}A{\h x}^{k+1})$.}
			\If{$(\alpha_k - \alpha)\|{\h z}^{k+1}\| > \varepsilon$}
				\State{Update {$\alpha_{k+1}$ according to \eqref{scalar} and \eqref{scalar2},}}
				\State{$\rho_{k+1} := {\frac{\ell_{\nabla  \varphi}}{2}} + \frac{3\gamma\sigma_A}{ 2(\alpha_{k+1} - \alpha)^2} + \mu$,}
				\State{$\delta_{k+1} := \frac{\alpha}{\alpha_{k+1}}$.}
           \Else
           \State{$\alpha_{k+1}:=\alpha_k$; $\rho_{k+1}:=\rho_k$; $\delta_{k+1}:=\delta_k$.}
		   \EndIf
		\State{{Update \(\beta_{k+1}\) dynamically to satisfy (\ref{betacond}).}}
            \EndFor
	\end{algorithmic}
\end{algorithm}
To overcome the lack of subsequence convergence of DFPS, we propose {Algorithm \ref{AdaptiveDPFS}}, an adaptive version of it, called Adaptive DPFS. This algorithm successfully addresses the decrease of the step sizes by incorporating a step size control strategy (see lines {5-11}). This strategy has two advantages: $(1)$ under suitable assumptions, it can take constant step sizes for all associated subproblems after a finite number of iterations, while the step sizes in the ${\h y}$-subproblem can also have a positive uniform lower bound, and $(2)$ the accuracy of the approximate criticality of the accumulation points of the structured DC program (\ref{PForm}) can be controlled a priori.

The following assumption ensures that the algorithm is well-defined.
 \begin{ass}  \label{ass:A(S)}
 Assume that $K(S) \subseteq {\rm ri}(\dom f)$ and there exists a scalar $r>0$ such that $\mathcal{K}_r := \{{\h z} \in \R^s: \dist({\h z}, A(S)) \leq r\} \subseteq \mathrm{int}(\dom g)$. We denote by $\ell_{g,\mathcal{K}_r}>0$ the Lipschitz constant of $g$ on $\mathcal{K}_r$.
\end{ass}

\begin{remark}\label{rem:equiv_Kr_int} In Assumption \ref{ass1} we assume that $A(S) \subseteq {\rm ri}(\dom g)$. If we further assume that ${\rm int}(\dom g)\neq \emptyset$, then ${\rm ri}(\dom g) = {\rm int}(\dom g)$. Since $A(S)$ is a compact set and ${\rm int}(\dom g)$ is an open set, there exists $r>0$ such that $\mathcal{K}_r\subseteq {\rm int}(\dom g)$. Thus ${\rm int}(\dom g)\neq \emptyset$ is equivalent to the existence of $r>0$ such that $\mathcal{K}_r := \{{\h z} \in \R^s: \dist({\h z}, A(S)) \leq r\} \subseteq \mathrm{int}(\dom g)$
\end{remark}

\begin{theorem}\label{descent_Adaptive} (Adaptive DPFS is well-defined) Suppose that  Assumption \ref{ass1} holds and Assumption \ref{ass:A(S)} is fulfilled with a pair of $r>0$ and $\ell_{g,\mathcal{K}_r}>0$. For $0<\gamma\le\frac{{2r}}{ \ell_{g,\mathcal{K}_r}}$ and $\norm{{\h z}^0} \leq  \ell_{g,\mathcal{K}_r}$, let $({\h x}^k, {\h y}^k, {\h z}^k )_{k\geq 0}$ be the sequence generated by Algorithm \ref{AdaptiveDPFS}, where the scalar sequences $(\alpha_{k})_{k \geq 0}$, $(\beta_{k})_{k \geq 0}$, $(\delta_{k})_{k \geq 0}$ and $(\rho_{k})_{k \geq 0}$ are chosen such that \eqref{scalar}-\eqref{rhocond} are satisfied. Then, the following statements are true:
\begin{itemize}
\item[{\rm (i)}] The sequence $({{\h z}^{k}})_{k \geq 0}$ is bounded.
 \item[{\rm (ii)}]
The inner loop of lines 6-8 in Algorithm \ref{AdaptiveDPFS} will be executed at most a finite number of times, so the algorithm is well-defined.
 \item[{\rm (iii)}] There exists an index $K_0$ such that for all $k\ge K_0$ it holds $\alpha_k = {\widetilde\alpha}>\alpha>0$,
$\rho_k = \widetilde \rho:={\frac{\ell_{\nabla \varphi}}{2}+\frac{3\sigma_A\gamma}{2(\widetilde\alpha-\alpha)^2}+\mu}$, $\delta_k = \frac{\alpha}{{\widetilde\alpha}}$, and
$(\alpha_k-\alpha)\| {\h z}^{k+1}\| \le \varepsilon$.
\end{itemize}
 \end{theorem}

 \begin{proof}
(i) We will prove this statement by induction. We have $\norm{{\h z}^0} \leq {\ell_{g,\mathcal{K}_r}}$ by the initialization of the algorithm and assume that for a $k\geq 0$, $\norm{{\h z}^k} \leq {\ell_{g,\mathcal{K}_r}}$ holds.  Since $\norm{\alpha {\h z}^{k}} \leq \alpha \ell_{g,\mathcal{K}_r} < r$, we have that ${A{\h x}^{k+1}+\alpha {\h z}^{k}}\in \mathcal{K}_r \subseteq \inte(\dom g)$. We choose ${\widetilde{\h z}}^{k+1}\in \partial g(A{\h x}^{k+1}+\alpha {\h z}^{k})$. By using again Assumption \ref{ass:A(S)}, we have that $\norm{\widetilde{\h z}^{k+1}}\leq  \ell_{g,\mathcal{K}_r}$.  By the optimality of the $\h{z}$-update in line $4$, we see that
\begin{align}
    & g^*({\h z}^{k+1}) - \skal{A{\h x}^{k+1}+\alpha {\h z}^k,{\h z}^{k+1}}+ \frac{\alpha_k}{2}\norm{{\h z}^{k+1}}^2 \notag \\
     \leq & \ g^*(\widetilde{\h z}^{k+1}) - \skal{A{\h x}^{k+1}+\alpha {\h z}^k,\widetilde{\h z}^{k+1}} + \frac{\alpha_k}{2}\norm{\widetilde{\h z}^{k+1}}^2 = - g(A{\h x}^{k+1} +\alpha {\h z}^k) + \frac{\alpha_k}{2}\norm{\widetilde{\h z}^{k+1}}^2 \notag \\
     \leq &  \ g^*({\h z}^{k+1}) - \skal{A{\h x}^{k+1}+\alpha {\h z}^k,{\h z}^{k+1}} + \frac{\alpha_k}{2}\norm{\widetilde{\h z}^{k+1}}^2. \notag
\end{align}
This implies that $\norm{{\h z}^{k+1}} \leq \norm{\widetilde{\h z}^{k+1}}\leq \ell_{g,\mathcal{K}_r}$ and the induction is completed.

 (ii) Suppose lines 6-8 are executed infinitely often, and let \((k_j)_{j \geq 0}\) be the ordered sequence of indices for which this happens. This implies that \((\alpha_{k_j} - \alpha) \|{\h z}^{k_j + 1}\| > \varepsilon\), and therefore \(\alpha_{k_j} - \alpha > \frac{\varepsilon}{{\ell_{g,\mathcal{K}_r}}} > 0\) for all $j \geq 0$. This is a contradiction to \(\lim_{k_j \to +\infty} (\alpha_{k_j} - \alpha) = 0\).

 (iii) According to (ii), there exists $K_0 \geq 0$ such that $(\alpha_k-\alpha)\norm{{\h z}^{k+1}}<\varepsilon$ for all $k\geq K_0$. The statements follow by defining $\widetilde{\alpha}:=\alpha_{K_0}$.
  \end{proof}
For $\widetilde{\alpha}$ the constant given by Theorem \ref{descent_Adaptive} (iii), we consider the following energy function
\begin{align}\label{varfun}
{\varpi}({\h x}, {\h y}, {\h z}, {\h u}) := & \ \Psi({\h x}, {\h y}, {\h z}) + \frac{\alpha}{2} \left( 1 + \frac{4\alpha}{\widetilde{\alpha} - \alpha} \right) \|{\h z} - {\h u}\|^2 - \frac{1}{2} (\widetilde{\alpha} - \alpha) \|{\h z}\|^2 \nonumber\\
= & \ \langle{\h z}, A{\h x} \rangle - g^*({\h z}) +\varphi({\h x}) + { \iota_S({\h x})} - \langle K{\h x},{\h y}\rangle +f^*({\h y})\\
& + \frac{\alpha}{2} \left( 1 + \frac{4\alpha}{\widetilde{\alpha} - \alpha} \right) \|{\h z} - {\h u}\|^2 - \frac{1}{2} (\widetilde{\alpha} - \alpha) \|{\h z}\|^2.\nonumber
\end{align}
Let $({\h x}^k, {\h y}^k, {\h z}^k )_{k\geq 0}$ be the sequence generated by Algorithm \ref{AdaptiveDPFS}.  For brevity we denote for $k\ge K_0+2$, where $K_0 \geq 0$ is the index also given by Theorem \ref{descent_Adaptive} (iii),
\begin{eqnarray*}
{{\varpi}}_{k}:={\varpi}({\h x}^{k},{\h y}^{k},{\h z}^{k},{\h z}^{k-1}).
\end{eqnarray*}

The following theorem establishes the descent property of Algorithm \ref{AdaptiveDPFS}.

 \begin{theorem}\label{Algo1:conv}
Suppose that  Assumption \ref{ass1} holds and  Assumption \ref{ass:A(S)} is fulfilled with a pair of $r>0$ and $\ell_{g,\mathcal{K}_r}>0.$ For $0<\gamma\le\frac{{2r}}{ \ell_{g,\mathcal{K}_r}}$ and $\norm{{\h z}^0} \leq  \ell_{g,\mathcal{K}_r}$, let $({\h x}^k, {\h y}^k, {\h z}^k )_{k\geq 0}$ be the sequence generated by Algorithm \ref{AdaptiveDPFS}, where the scalar sequences $(\alpha_{k})_{k \geq 0}$, $(\beta_{k})_{k \geq 0}$, $(\delta_{k})_{k \geq 0}$ and $(\rho_{k})_{k \geq 0}$ are chosen such that \eqref{scalar}-\eqref{rhocond} are satisfied. Let $K_0 \geq 0$ and ${\widetilde \alpha}>\alpha>0$ be as given by Theorem \ref{descent_Adaptive} (iii). Then, the following statements are true:
\begin{itemize}
\item[(i)] For all $k\ge K_0+2$ it holds
 \begin{eqnarray}\label{descent2}
 &&{{\varpi}}_{k+1}\le {{\varpi}}_{k}-c_1\|{\h x}^k-{\h x}^{k+1}\|^2
 -c_2\|{\h y}^k-{\h y}^{k+1}\|^2
-c_3\|{\h z}^{k+1}-{\h z}^{k}\|^2,\nn\\
\end{eqnarray}
where $c_1,\;c_2,\;c_3>0$.
\item[(ii)] There exists $\underline{\varpi} \in \R$ such that
$  \lim\limits_{k \to +\infty} {{\varpi}}_{k}=\underline{\varpi}.$
\item[(iii)] The sequences of discrete velocities vanish, i.e.,
\begin{eqnarray*} \mathop{\lim}\limits_{k \to +\infty} \|{\h x}^k-{\h x}^{k+1}\|=0, \quad \mathop{\lim}\limits_{k \to +\infty} \|{\h y}^k-{\h y}^{k+1}\|=0, \quad \mathop{\lim}\limits_{k \to +\infty} \|{\h z}^k-{\h z}^{k+1}\|=0.\end{eqnarray*}
\item[(iv)] If ${{{\rm int}(\dom f)}} \neq \emptyset$, then the sequence $( {\h x}^k,{\h y}^k, {\h z}^k)_{k\geq 0}$ is bounded.
\end{itemize}
 \end{theorem}
 \begin{proof}(i) By combining Theorem \ref{PriTheo} with Theorem \ref{descent_Adaptive}, and
 setting $c_1=\mu$, $c_2=\underline{\beta}$ and $c_3={\alpha(1-\frac{2\alpha}{\gamma})}$, the conclusion follows directly from \eqref{descent}, since $\omega_k = 0$ for $k\geq K_0+2$.

 (ii)-(iii) The sequence of $({{\varpi}}_{k})_{k\ge K_0+2}$ is lower bounded by a similar argument as in the proof of Theorem \ref{subsequentialtoStat} (i). The existence of the limit and the vanishing of the discrete velocities follow by Lemma \ref{ak}.

(iv) Since $S$ is a nonempty and compact set,  we deduce that $({\h x}^k)_{k\geq 0}$ is bounded. From line 3 of Adaptive DFPS,  we see that for all $k \geq 0$
    \begin{align*}
        {{\h y}}^{k+1} &\in \partial f(K {{\h x}}^{k+1} + \beta_k ({{\h y}}^{k}-{{\h y}}^{k+1}) ).
    \end{align*}
Since ${\rm int}(\dom f)$ is nonempty and $K(S) \subseteq {\rm ri}(\dom f) = \inte(\dom f)$, there exists $\delta>0$ such that $\{{\h y} \in \R^p: \dist({\h y},K(S)) \leq \delta \} \subseteq {\rm int}(\dom f)$.  Since $\|{\h y}^k-{\h y}^{k+1}\| \to 0$ as $k \to +\infty$ and $(\beta_k)_{k \geq 0}$ is a bounded sequence, and the convex subdifferential is locally bounded in the interior of the domain, it follows that $({\h y}^k)_{k\geq 0}$ is bounded. The boundedness of $({{\h z}^{k}})_{k \geq 0}$ follows by Theorem \ref{descent_Adaptive} (iii).
 \end{proof}

\begin{theorem}\label{descent_Adaptivesol} (Subsequential convergence of Adaptive DPFS) 
Suppose that  Assumption \ref{ass1} holds and  Assumption \ref{ass:A(S)} is fulfilled with a pair of $r>0$ and $\ell_{g,\mathcal{K}_r}>0.$ For $0<\gamma\le\frac{{2r}}{ \ell_{g,\mathcal{K}_r}}$ and $\norm{{\h z}^0} \leq  \ell_{g,\mathcal{K}_r}$, let $({\h x}^k, {\h y}^k, {\h z}^k )_{k\geq 0}$ be the sequence generated by Algorithm \ref{AdaptiveDPFS}, where the scalar sequences $(\alpha_{k})_{k \geq 0}$, $(\beta_{k})_{k \geq 0}$, $(\delta_{k})_{k \geq 0}$ and $(\rho_{k})_{k \geq 0}$ are chosen such that \eqref{scalar}-\eqref{rhocond} are satisfied. Let $0<\varepsilon\leq r$ and $\ell_{g,\mathcal{K}_{\varepsilon}}>0$ be the Lipschitz constant of $g$ on the compact set $\mathcal{K}_{\varepsilon} =  \{{\h z} \in \R^s : \dist({\h z},A(S))\leq \varepsilon \}\subseteq{\text{\rm int}}(\dom g)$. Then, every accumulation point of the sequence $({\h x}^k,{\h y}^k,{\h z}^k)_{k\ge 0}$ is a $(2\ell_{g,\mathcal{K}_{\varepsilon}}\varepsilon)$-approximate KKT point of (\ref{PForm}).
\end{theorem}
\begin{proof} According to (iii) of Theorem \ref{descent_Adaptive}, there exists a $K_0 \geq 0$ such that  $\alpha_k\equiv {\widetilde\alpha}>\alpha>0$,
$\rho_k \equiv \rho$ and $\delta_k\equiv \alpha/{\widetilde\alpha}$ for all $k\ge K_0$.
Invoking the lines 2-4 in Adaptive DFPS, for all $k\ge K_0+2$ we have that
\begin{eqnarray}\label{algoKKT}\left\{\begin{array}{l}{\bf 0}\in  N_S({\h x}^{k+1})+ A^* {\h z}^k -K^* {\h y}^k +\nabla \varphi({\h x}^k)+\rho({\h x}^{k+1}-{\h x}^k),\\[0.1cm]
{\bf 0}\in \partial f^*({\h y}^{k+1}) -K{\h x}^{k+1} +\beta_k({\h y}^{k+1}-{\h y}^k),\\[0.1cm]
A{\h x}^{k+1} -{\widetilde\alpha}{\h z}^{k+1}+\alpha {\h z}^k\in\partial g^*({\h z}^{k+1}).\end{array}\right.
\end{eqnarray}
Let $(\overline{\h x},\overline{\h y},\overline{\h z})$ be an accumulation point of  $({\h x}^k,{\h y}^k,{\h z}^k)_{k\ge 0}$, and  $({\h x}^{k_j},{\h y}^{k_j},{\h z}^{k_j})_{j\ge 0}$ be a subsequence converging to $(\overline{\h x},\overline{\h y},\overline{\h z})$ as $j \to +\infty$ with
${\beta}_{k_j}\to {\beta}$ as $j \to +\infty$ (due to the boundedness of $(\beta_{k})_{k \ge 0}$). Using Theorem \ref{Algo1:conv} (iii), we have
\begin{eqnarray}\label{KKT} \left\{\begin{array}{l} {\bf 0}\in N_S(\overline{\h x}) +A^* \overline{\h z} -K^* \overline{\h y} +\nabla \varphi(\overline{\h x}),\\[0.1cm]
K{\overline{\h x}}\in \partial f^*(\overline{\h y}),\\[0.1cm]
A\overline{\h x} -({\widetilde\alpha}-\alpha)\overline{\h z}\in\partial g^*(\overline{\h z}).\end{array}\right.
\end{eqnarray}
By Theorem \ref{descent_Adaptive} (ii), we have that
$(\alpha_k-\alpha)\|{\h z}^{k+1}\| < \varepsilon$ for all $k\ge K_0$. This yields
$({\widetilde \alpha}-\alpha)\|{\overline{\h z}}\|\le \varepsilon.$
It follows from the third inclusion in (\ref{KKT}) and the above inequality that
$${\text{dist}}(\partial g^*({\overline{\h z}}), A{\overline{\h x}})\le \varepsilon.$$
From here, by Lemma \ref{epsappsol}, we have that
$\overline{\h z}\in \partial_{\widehat\varepsilon} g(A\overline{\h x})$ with
${\widehat\varepsilon}=2{  \ell_{g,\mathcal{K}_{\varepsilon}}} \varepsilon$ and thus by \eqref{KKT} we conclude the result.
 \end{proof}
Let $\Omega({\h x}^0, {\h y}^0, {\h z}^0)$ denote the set of accumulation points of $({\h x}^k, {\h y}^k, {\h z}^k)_{k \geq 0}$.
\begin{lemma}\label{constantO}
 Under the same assumptions as in Theorem \ref{Algo1:conv}, let $({\h x}^k, {\h y}^k, {\h z}^k)_{k \geq 0}$ be the sequence generated by Algorithm \ref{AdaptiveDPFS}. Additionally, assume ${\text{\rm int}}(\dom f) \neq \emptyset$. Then, the following statements are true:
\begin{itemize}
\item[(i)] $\Omega({\h x}^0, {\h y}^0, {\h z}^0)$ is a nonempty, compact and connected set;
 \item[(ii)] For every $({\overline{\h x}},{\overline{\h y}},{\overline{\h z}})\in\Omega({\h x}^0, {\h y}^0, {\h z}^0)$, it holds ${{\varpi}}({\overline{\h x}},{\overline{\h y}},{\overline{\h z}},{\overline{\h z}}) = \underline{\varpi}$.
 \item[(iii)] It holds $\dist \big(({\h x}^k, {\h y}^k, {\h z}^k), \Omega({\h x}^0, {\h y}^0, {\h z}^0)\big) \to 0$ as $k\to +\infty$.
\end{itemize}
\end{lemma}
\begin{proof} According to Theorem \ref{Algo1:conv} {(iv)}, the sequence $( {\h x}^k,{\h y}^k, {\h z}^k)_{k\geq 0}$ is bounded.

(i) The proof can be carried out in a similar manner to that of \cite[Lemma 5 (iii)]{BST14}.

(ii) Let $({\overline{\h x}},{\overline{\h y}},{\overline{\h z}})\in\Omega({\h x}^0, {\h y}^0, {\h z}^0)$ and consider the subsequence $( {\h x}^{k_j},{\h y}^{k_j}, {\h z}^{k_j})_{j \geq 0}$ with the property that
$\lim_{j\to +\infty} ( {\h x}^{k_j},{\h y}^{k_j}, {\h z}^{k_j}) = ({\overline{\h x}},{\overline{\h y}},{\overline{\h z}}).$
From Theorem \ref{Algo1:conv} we have that
$\lim_{j\to+\infty}{{{{{\varpi}}}}}_{k_j}= \lim_{j\to+\infty}{{{{\varpi}}}}({\h x}^{k_j},{\h y}^{k_j},{\h z}^{k_j},{\h z}^{{k_j}-1})=\underline{{{{\varpi}}}}$.

Using the update in line $3$ of the algorithm,  yields for all $k_j \geq K_0+2$
\begin{align*}
  \varpi_{k_j+1} +\frac{\beta_{k_j}}{2}\|{\h y}^{k_j+1}-{\h y}^{k_j}\|^2\leq \,\varpi({\h x}^{k_j+1},\overline{\h y},{\h z}^{k_j+1},{\h z}^{{k_j}})+\frac{\beta_{k_j}}{2}\|{\overline{\h y}}-{\h y}^{k_j}\|^2.
\end{align*}
Since ${\h x}^{k_j}\in S$ for all $k_j \geq K_0+2$ and the function $g^*$ is lower semicontinuous, Theorem \ref{Algo1:conv} (iii) and condition (\ref{betacond}) imply that
\begin{align*}
    \underline{\varpi} \leq \limsup_{j \to +\infty }\left( \varpi({\h x}^{k_j+1},\overline{\h y},{\h z}^{k_j+1},{\h z}^{{k_j}})+\frac{\beta_{k_j}}{2}\|{\overline{\h y}}-{\h y}^{k_j}\|^2 \right)\leq \varpi(\overline{\h x}, \overline{\h y}, \overline{\h z}, \overline{\h z}).
\end{align*}

Using the update in line $4$ of the algorithm,  yields for all $k_j \geq K_0+2$
\begin{align*}\Psi({\h x}^{k_j+1},{\h y}^{k_j+1},{\h z}^{k_j+1})- \frac{\widetilde\alpha}{2}\left\|{\h z}^{k_j+1}-\frac{\alpha}{\widetilde \alpha} {\h z}^{k_j}\right\|^2
 \ge  \Psi({\h x}^{k_j+1},{\h y}^{k_j+1},\overline{\h z})-  \frac{\widetilde\alpha}{2}\left\|{\overline{\h z}}-\frac{\alpha}{\widetilde \alpha} {\h z}^{k_j}\right\|^2
 \end{align*}
or, equivalently,
\begin{align*}
    & \ \Psi({\h x}^{k_j+1},{\h y}^{k_j+1},{\h z}^{k_j+1}) -\frac{\alpha}{2}\norm{{\h z}^{k_j+1}- {\h z}^{k_j}}^2-\frac{\widetilde{\alpha}-\alpha}{2}\norm{{\h z}^{k_j+1}}^2\\
 \ge & \ \Psi({\h x}^{k_j+1},{\h y}^{k_j+1},\overline{\h z})-\frac{\alpha}{2}\norm{\overline{\h z}- {\h z}^{k_j}}^2-\frac{\widetilde{\alpha}-\alpha}{2}\norm{\overline{\h z}}^2.
\end{align*}
Since  ${\h x}^{k_j}\in S$ for all $k_j \geq K_0+2$ and $f^*$ is lower semicontinuous, Theorem \ref{Algo1:conv} (iii) leads to
\begin{align*}
    \underline{\varpi} \geq \liminf_{j\to +\infty} \left( \Psi({\h x}^{k_j+1},{\h y}^{k_j+1},\overline{\h z})-\frac{\alpha}{2}\norm{\overline{\h z}- {\h z}^{k_j}}^2-\frac{\widetilde{\alpha}-\alpha}{2}\norm{\overline{\h z}}^2\right) \geq \varpi(\overline{\h x}, \overline{\h y}, \overline{\h z}, \overline{\h z}).
\end{align*}
From here, (ii) follows immediately. 

(iii) This is a direct consequence of the fact that every bounded sequence approaches its set of accumulation points.
\end{proof}

\section{Global convergence under the Kurdyka-\L ojasiewicz property}\label{sec:GlobalConv}

In this section, we discuss the convergence of the sequence $({\h x}^k, {\h y}^k, {\h z}^k)_{k \geq 0}$ generated by Algorithm \ref{AdaptiveDPFS} in the framework of the Kurdyka-\L ojasiewicz (KL) property.

The framework in which we are going to do this is set out in the following assumption.

\begin{ass}\label{ass2} Assume that \( g \) \textit{essentially strictly convex} and  ${\text{\rm int}}(\dom f) \neq \emptyset$. \end{ass}

We state the definition of the KL property.

\begin{definition} \label{def:KL}
A proper function $h: \mathbb R^n\rightarrow {\overline{\mathbb R}}$ is said to satisfy the Kurdyka-\L ojasiewicz (KL) property at a point ${\widehat{\h x}}\in {\text{\rm dom }}\partial h$ if there exist a constant $\eta\in(0,+\infty]$, an open neighborhood
	$U$ of ${\widehat{\h x}},$ and a concave and continuous function $\phi:\;[0,\eta)\rightarrow[0,+\infty)$ with the properties
	\begin{itemize}
        \item[(i)] $\phi(0)=0$;
		\item[(ii)] $\phi$ is continuously differentiable on $(0,\eta)$ and $\phi'{(\cdot)}>0$ on $(0,\eta)$;
		\item[(iii)] for every $\h x\in U$ with $h({\widehat{\h x}})< h(\h x)< h(\widehat{\h x}) +\eta$, it holds
		\begin{eqnarray*}\phi'(h(\h x)-h(\widehat{\h x})){\text{\rm{dist}}}(\h 0,\partial h({\h x}))\ge 1. \end{eqnarray*}
	\end{itemize}
The class of so-called desingularization functions $\phi:[0,\eta)\rightarrow[0,+\infty)$ fulfilling (i)-(iii) is denoted by $\Phi_\eta$. 
\end{definition}

\begin{remark} \label{rem:local_lsc}
The classical Definition~\ref{def:KL} of the Kurdyka--Łojasiewicz (KL) property assumes the function $h$ is \emph{proper and lower semicontinuous} on the whole
space. In the subsequent analysis we would like to assume that the merit function $\varpi$ satisfies the KL property, however, since it contains the concave term
$-g^{\ast}$, it can fail to be everywhere lower semicontinuous. On the other hand, by Assumption~\ref{ass2}, $-g^{\ast}$ is differentiable on
$\mathrm{int}(\dom g^*)=\dom\partial g^{\ast}$, therefore,
\begin{align*} 
& \ \varpi \ \text{is lower semicontinuous on } \R^n \times \R^p \times \mathrm{int}(\dom g^*) \times \R^s \\ \mbox{and} & \ \dom\partial\varpi\subseteq \R^n \times \R^p \times \mathrm{int}(\dom g^*) \times \R^s.
\end{align*}
Therefore, when assuming in the following that $\varpi$ satisfies the KL property at a point $\widehat {\h w}:=(\widehat{\h x},\widehat{\h y},\widehat{\h z},\widehat{\h u}) \in \dom\partial\varpi$, in the following we will restrict the open neighborhood $U$ of $\widehat{\h w} $ such that $U \subseteq \R^n \times \R^p \times \mathrm{int}(\dom g^*) \times \R^s$, on which $\varpi$ is lower semicontinuous.
\end{remark}

The class of functions that satisfy the KL includes semialgebraic, real subanalytic, uniformly convex functions and convex functions satisfying a growth condition (\cite{ABRS,ABS,BST14, BCN19}).

The following lemma can be found in \cite{BST14} and provides a uniform version of the KL property on a compact set.

\begin{lemma}\label{uniform} Let $\Omega$ be a compact set and $h:\mathbb R^n \rightarrow {\overline{\mathbb R}}$ a proper function. Assume that $h$ is constant on $\Omega$, lower semicontinuous on an open set containing $\Omega$, and it satisfies the KL property at each point $\Omega$. Then, there exist $\varrho>0$ and $\eta >0$,
$\phi\in\Phi_{\eta}$ such that for every $\widehat{\h x}\in \Omega$ and
every element ${\h x}$ in the intersection
\begin{eqnarray*}\{\h x \in \mathbb R^n :{\text{\rm dist}}({\h x},\Omega)<\varrho\}\cap\{{\h x} \in \mathbb R^n: h(\widehat{\h x})<h({\h x})<h(\widehat{\h x})+\eta\} \end{eqnarray*}
it holds
\begin{eqnarray*} \phi'(h({\h x})-h(\widehat{\h x})){\text{\rm dist}}({\h 0},\partial {h}(\h x))\ge 1.\end{eqnarray*}
\end{lemma}

As we already established in Theorem \ref{Algo1:conv}, if the assumptions \ref{ass1} and \ref{ass:A(S)} hold and ${\text{\rm int}}(\dom f) \neq \emptyset$,  then the sequence $({\h x}^k,{\h y}^k,{\h z}^k)_{k\ge 0}$ is bounded.

 \begin{theorem}\label{residual}
 \!\! Under the same assumptions as in Theorem \ref{Algo1:conv}, let \!$({\h x}^k, {\h y}^k, {\h z}^k)_{k \geq 0}$ be the sequence generated by Algorithm \ref{AdaptiveDPFS}. In addition, let Assumption \ref{ass2} hold.
 Then there exists  a constant $\zeta>0$ such that for all $k\ge K_0+2$
\begin{align*} & \ {\text{\rm dist}} \Big (\h 0,\partial{\varpi}  ({\h x}^{k+1},{\h y}^{k+1},{\h z}^{k+1},{\h z}^k) \Big )\\
\le & \ \zeta(\|{\h x}^k-{\h x}^{k+1}\|+\|{\h z}^k -{\h z}^{k+1}\|+\|{\h y}^k-{\h y}^{k+1}\|).\end{align*}
\end{theorem}

\begin{proof} %
We are going to construct an element in ${\partial} {\varpi}({\h x}^{k+1},{\h y}^{k+1},{\h z}^{k+1},{ {\h u}^{k}})$.
Notice that by the sum rule \cite[Example 8.8]{RockWets} we have:
\begin{align*}
 \left\{
 \begin{array}{l}
 {\partial}_{\h x} {\varpi}({\h x},{\h y},{\h z},{ {\h u}}) = N_S({\h x}) + A^*{\h z}+\nabla \varphi({\h x})-K^*{{\h y}}\\[0.1cm]
 {\partial}_{\h y} {\varpi}({\h x},{\h y},{\h z},{ {\h u}}) = -K{{\h x}}+\partial f^*({\h y})\\[0.1cm]
 {\partial}_{\h z} {\varpi}({\h x},{\h y},{\h z},{ {\h u}}) = A{\h x}+\partial(-g^*)({\h z})+\alpha \left(1+\frac{4\alpha}{{\widetilde \alpha}-\alpha}\right)({\h z}-{{\h u}})-({\widetilde \alpha}-\alpha){\h z}\\[0.1cm]
 {\partial}_{\h u} {\varpi}({\h x},{\h y},{\h z},{ {\h u}})=  \alpha\left(1+\frac{4\alpha}{{\widetilde \alpha}-\alpha}\right)({{\h u}}-{\h z}).
 \end{array}
 \right.
 \end{align*}

Let $k \geq K_0+2$.  Using the update in line $2$ of the algorithm, there exists a vector $\upsilon^{k+1}\in N_{S}(\h x^{k+1})$
 such that
 \begin{eqnarray*}
 0= \rho\upsilon^{k+1}+\nabla \varphi({\h x}^k)-K^* {\h y}^k+A^*{\h z}^k+\rho({\h x}^{k+1}-{\h x}^k).\end{eqnarray*}
We have ${\pmb{\xi}_{\h x}^{k+1}}:=\rho\upsilon^{k+1}+ A^*{\h z}^{k+1}+\nabla \varphi({\h x}^{k+1})-K^*{{\h y}^{k+1}} \in \partial_{\h x}{\varpi}({\h x}^{k+1},{\h y}^{k+1},{\h z}^{k+1},{\h z}^{k})$, and thus
\begin{align*}
   \norm{{\pmb{\xi}_{\h x}^{k+1}}}&=\norm{A^*({\h z}^{k+1}-{\h z}^k)-\rho({\h x}^{k}-{\h x}^{k+1})+\nabla \varphi({\h x}^{k+1})-\nabla \varphi({\h x}^k)+K^*({{\h y}^{k}}-{\h y}^{k+1})}  \\
   &\leq\sqrt{\sigma_A}\|{\h z}^k-{\h z}^{k+1}\| + (\rho+\ell_{\nabla \varphi})\|{\h x}^{k}-{\h x}^{k+1}\|+\|K\|\|{\h y}^k-{\h y}^{k+1}\|.
\end{align*}
 Using the update in line 3 of the algorithm, there exists $v^{k+1}\in \partial f^*({\h y}^{k+1})$ such that
\begin{eqnarray*}
0 =  v^{k+1}-K{\h x}^{k+1}+\beta_k({\h y}^{k+1}-{\h y}^k).\end{eqnarray*}
For ${\pmb{\xi}_{\h y}^{k+1}}:=v^{k+1}-K{\h x}^{k+1} \in \partial_{\h y}{\varpi}({\h x}^{k+1},{\h y}^{k+1},{\h z}^{k+1},{\h z}^{k})$, we have
\begin{align*}
    \norm{{\pmb{\xi}_{\h y}^{k+1}}}=\beta_k\norm{{\h y}^{k+1}-{\h y}^k} \leq \overline{\beta}\|{\h y}^k-{\h y}^{k+1}\|.
\end{align*}
From the update in line 4 of the algorithm, we have that
\begin{align*}
    0\in \partial  g^*({\h z}^{k+1})-A{\h x}^{k+1}+\widetilde{ \alpha} {\h z}^{k+1} -{\alpha} {\h z}^k.
\end{align*}
Thus ${\h z}^{k+1} \in \dom \partial g^*$. Since $g^*$ is essentially smooth, we have $\dom \partial g^*=\mathrm{int}(\dom g^*)$ \cite{EssentialSmooth}. Therefore it follows that $g^*$ is differentiable at ${\h z}^{k+1}$. Consequently, for ${\pmb{\xi}_{\h z}^{k+1}}:= A{\h x}^{k+1}-\nabla g^*({\h z}^{k+1})+\alpha (1+\frac{4\alpha}{{\widetilde \alpha}-\alpha})({\h z}^{k+1}-{{\h z}^{k}})-({\widetilde \alpha}-\alpha){\h z}^{k+1} \in \partial_{\h z}{\varpi}({\h x}^{k+1},{\h y}^{k+1},{\h z}^{k+1},{\h z}^{k}) $ it holds
 $$\|\pmb{\xi}_{\h z}^{k+1}\|\le \alpha \left(2+\frac{4\alpha}{{\widetilde \alpha}-\alpha}\right)\|{\h z}^k-{\h z}^{k+1}\|.$$

Finally, for ${\pmb{\xi}_{{\h {u}}}^{k+1}}:=\alpha({1}+\frac{4\alpha}{{\widetilde \alpha}-\alpha})({\h z}^k-{\h z}^{k+1})\in \partial_{{\h  {u}}}{\varpi}({\h x}^{k+1},{\h y}^{k+1},{\h z}^{k+1},{\h z}^{k})$, we also have
 $\|\pmb{\xi}_{{\h u}}^{k+1}\|\le \alpha \left({1}+\frac{4\alpha}{{\widetilde \alpha}-\alpha}\right)\|{\h z}^k-{\h z}^{k+1}\|.$

 Since
 \begin{align*}
     {\text{\rm dist}}\Big (\h 0,\partial{\varpi} ({\h x}^{k+1},{\h y}^{k+1},{\h z}^{k+1},{\h z}^k)\Big) \leq \norm{{\pmb{\xi}_{\h x}^{k+1}}} + \norm{{\pmb{\xi}_{\h y}^{k+1}}} + \norm{{\pmb{\xi}_{\h z}^{k+1}}} + \norm{{\pmb{\xi}_{{\h u}}^{k+1}}},
 \end{align*}
the conclusion follows.
\end{proof}
The convergence of the whole sequence is provided by the following theorem, whose proof follows the spirit of \cite[Theorem 4]{LiPong15} and \cite[Theorem 3.4]{BN20}.

 \begin{theorem} \label{thm_conv}
 Under the same assumptions as Theorem \ref{Algo1:conv}, let $({\h x}^k, {\h y}^k, {\h z}^k)_{k \geq 0}$ be generated by Algorithm \ref{AdaptiveDPFS}. In addition, let Assumption \ref{ass2} hold. Let $0<\varepsilon\leq r$ and $\ell_{g,\mathcal{K}_{\varepsilon}}>0$ be the Lipschitz constant of $g$ on the compact set $\mathcal{K}_{\varepsilon} =  {\{{\h z} \in \R^s : \dist({\h z},A(S))\leq \varepsilon \}\subseteq{\text{\rm int}}(\dom g)}$. Suppose that \({{\varpi}}\), as defined in (\ref{varfun}), satisfies the KL property at every point $({\h x}, {\h y}, {\h z}, {\h u}) \in \dom \partial \varpi$. Then,
$({\h x}^k, {\h y}^k, {\h z}^k)_{k \geq 0}$ converges to a $(2\ell_{g,\mathcal{K}_{\varepsilon}}\varepsilon)$-approximate KKT point of (\ref{PForm}).
\end{theorem}
\begin{proof}
As seen in Theorem \ref{Algo1:conv}, the sequence $( {{\varpi}}_{k})_{k \geq K_0+2}$ is nonincreasing and its limit exists and is equal to
$\underline{{\varpi}}$.  We divide the proof into two cases.

\textit{Case I}: We assume that there exists $\widehat K\ge K_0+2$  such that ${{\varpi}}_{\widehat K}= \underline{{\varpi}}$. Then, by (\ref{descent2}),
$({\h x}^{k+1}, {\h y}^{k+1}, {\h z}^{k+1}) = ({\h x}^k, {\h y}^k, {\h z}^k)$ for all $k\ge {\widehat K}$. Thus, the conclusion follows directly from Theorem \ref{descent_Adaptivesol}.

\textit{Case II}: It holds ${{\varpi}}_{k+1}>\underline{{\varpi}}$ for all $k \geq K_0+2$. From Theorem \ref{constantO} we know that $\Omega({\h x}^0, {\h y}^0,{\h z}^0)$ is nonempty and compact, and that $\varpi(\overline{\h x}, \overline{\h y}, \overline{\h z}, \overline{\h z}) = \underline{\varpi}  $ for all $(\overline{\h x}, \overline{\h y}, \overline{\h z}) \in \Omega({\h x}^0, {\h y}^0,{\h z}^0)$. Considering Remark \ref{rem:local_lsc}, by Lemma \ref{uniform} there exist $\varrho>0$, $\eta>0$ and $\phi \in \Phi_\eta$ such that for every $(\h x, \h y, \h z, \h u)$ in
\begin{align*}
    \Lambda := & \Big \{ ({\h x}, {\h y}, {\h z}, {\h u})\;:\; {\text{dist}}\big(({\h x}, {\h y}, {\h z}),\Omega ({\h x}^0,{\h y}^0,{\h z}^0) \big)<\varrho, \norm{{\h u} - {\h z}}<\varrho,\\
    & \qquad \qquad \qquad \qquad \qquad \qquad \qquad \underline{{\varpi}}<{{\varpi}} ({\h x}, {\h y}, {\h z}, {\h u}) <\underline{{\varpi}}+\eta \Big\} \\
    \subseteq &  \R^n \times \R^p \times \mathrm{int}(\dom g^*) \times \R^s
\end{align*}
the following holds
\begin{align*}
    \phi'({{\varpi}}({\h x},{\h y},{\h z},{\h u}) - \underline{{\varpi}}){\text{\rm{dist}}}(\h 0,\partial{{\varpi}} ({\h x},{\h y},{\h z},{\h u}))\ge 1.
\end{align*}
By Theorem \ref{Algo1:conv} and  {Lemma} \ref{constantO} (iii), there exists an index $\widehat K\ge K_0+2$ with the property that
 $({\h x}^k,{\h y}^k,{\h z}^k,{\h z}^{k-1})\in {\rm{\bm\Lambda}}$ for all  $k\ge\widehat K$. Thus, according to (\ref{descent2}), it holds for all $k \geq \widehat K$
\begin{align*}
\phi({{\varpi}}_k - \underline{{\varpi}}) -\phi({{\varpi}}_{k+1}-\underline{{\varpi}}) & \ge \phi'({{\varpi}}_k - \underline{{\varpi}})\Big({{\varpi}}_k - {{\varpi}}_{k+1} \Big)\nn\\
& \ge \frac{c}{{\text{\rm{dist}}}(\h 0, \partial{{\varpi}} ({\h x}^{k},{\h y}^{k},{\h z}^{k},{\h z}^{k-1}))}\Delta({\h x}^k,{\h y}^k,{\h z}^k), \end{align*}
where $c:=\min(c_1, c_2, c_3) >0$ and $\Delta({\h x}^k,{\h y}^k,{\h z}^k):=\|{\h x}^k-{\h x}^{k+1}\|^2+\|{\h y}^k-{\h y}^{k+1}\|^2+\|{\h z}^k-{\h z}^{k+1}\|^2$.

By denoting $\delta_{{\h x},{\h y},{\h z}}^{k}:=\|{\h x}^{k-1}-{\h x}^{k}\|+\|{\h y}^{k-1}-{\h y}^{k}\|+\|{\h z}^{k-1}-{\h z}^{k}\|$, the previous inequality and Theorem \ref{residual} lead for all $k \geq \widehat K$ to
\begin{eqnarray*}  \frac{1}{3}(\delta_{{\h x},{\h y},{\h z}}^{k+1})^2&&\le \Delta({\h x}^k,{\h y}^k,{\h z}^k)\nn\\
&&\le \frac{{\text{\rm{dist}}}(\h 0,\partial{{\varpi}} ({\h x}^{k},{\h y}^{k},{\h z}^{k},{\h z}^{k-1}))}{c}\Big(\phi({{\varpi}}^k - \underline{{\varpi}}) -\phi({{\varpi}}^{k+1}-\underline{{\varpi}}) \Big)\nn\\
&&\le \frac{\zeta }{c}\delta_{{\h x},{\h y},{\h z}}^{k}\Big(\phi({{\varpi}}_k - \underline{{\varpi}}) -\phi({{\varpi}}_{k+1}-\underline{{\varpi}})\Big).
\end{eqnarray*}
The AM-GM inequality yields that for all $k \geq \widehat K$
\begin{eqnarray}\label{keydesc}&&2\delta_{{\h x},{\h y},{\h z}}^{k+1}\le \delta_{{\h x},{\h y},{\h z}}^k+  \frac{3\zeta }{c} \Big(\phi({{\varpi}}_k - \underline{{\varpi}}) -\phi({{\varpi}}_{k+1}-\underline{{\varpi}})\Big). \end{eqnarray}
Consequently,
$ \sum_{k=1}^{+\infty} \delta_{{\h x},{\h y},{\h z}}^k<+\infty.$ In other words, $({\h x}^k, {\h y}^k, {\h z}^k)_{k \geq 0}$ is a Cauchy sequence and therefore convergent. The conclusion follows from Theorem \ref{descent_Adaptivesol}.
\end{proof}

\section{Convergence rates} \label{sec:ConvRates}

In this section, we analyze the convergence rates of the sequence $({\h x}^k,{\h y}^k,{\h z}^k)_{k\ge 0}$ generated by Algorithm \ref{AdaptiveDPFS} and for the sequence $(\varpi^k)_{k \geq 0}$. The following lemma will play a crucial role in our analysis.  For its proof we refer to \cite[Lemma 3.6]{BCN19} and \cite[Lemma 10]{BN20}.

\begin{lemma}\label{scalarrate}
Let $(e_k)_{k\ge0}$ be a nonincreasing sequence of nonnegative real numbers that converges to $0$. Assume that there exists an index $k_0 \geq 1$ such that for all $k\ge k_0$,
\begin{eqnarray*}e_{k-1} -e_k \ge c e_{k}^{2\theta},\end{eqnarray*}
where $c>0$ is some constant and $\theta\in[0,1)$. Then, the following statements are true:
\begin{itemize}
\item[(i)] if $\theta=0$, then $(e_k)_{k\ge 0}$ converges in finite time;
\item[(ii)] if $\theta\in (0,\frac{1}{2}]$, then there exist $\kappa_1>0$ and $q\in [0,1)$  such that for all $k\ge k_0$
\begin{eqnarray*}0\le e_k\le \kappa_1 q^k.\end{eqnarray*}

\item[(iii)] if $\theta\in(\frac{1}{2},1)$, then there exists $\kappa_2>0$ such that for all $k\ge k_0+1$
\begin{eqnarray*}0\le e_k\le \kappa_2 k^{\frac{1}{1-2\theta}}.\end{eqnarray*}
\end{itemize}
\end{lemma}

\begin{theorem}\label{congrate}
Under the same assumptions as {Theorem \ref{thm_conv}}, let $({\h x}^k, {\h y}^k, {\h z}^k)_{k \geq 0}$ be generated by Algorithm \ref{AdaptiveDPFS}. Suppose that there exists $\theta \in [0,1)$ such that $\varpi$ satisfies the KL property at every $({\h x}, {\h y}, {\h z}, {\h u}) \in \dom \partial \varpi $ with desingularization function $\phi(s)=\frac{\sigma}{1-\theta}s^{1-\theta}$ for some $\sigma > 0$. Then, the following statements are true:

\begin{itemize}
\item[(i)] If $\theta=0$, the ${\varpi}_{k}$ converges in finite time to $\underline{\varpi}$ as $k \to +\infty$;
\item[(ii)] If $\theta\in(0,\frac{1}{2}]$, then there exist $k_0\ge K_0+2$ and $q\in(0,1)$, and a constant $\kappa_1>0$ such that $0\le {{\varpi}}_{k}-\underline{\varpi} \le \kappa_1 q^k$ for all $k\ge k_0$;
\item[(iii)] If $\theta\in(\frac{1}{2},1)$, then there exist $k_0\ge K_0+2$  and a constant $\kappa_2>0$ such that 
$0 \le {{\varpi}}_{k}-\underline{\varpi} \le \kappa_2  k^{-\frac{1}{2\theta-1}}$ for all $k\ge k_0+1$.
\end{itemize}
\end{theorem}

\begin{proof}
From Theorem  \ref{Algo1:conv} and Theorem \ref{residual}, we have that for all $k \geq K_0+2$
\begin{eqnarray}({{\varpi}}_{k}-\underline{\varpi}) -({{\varpi}}_{k+1}-\underline{\varpi})&\ge& {c}\left(\|{\h x}^k-{\h x}^{k+1}\|^2
 +\|{\h y}^k-{\h y}^{k+1}\|^2
+\|{\h z}^{k+1}-{\h z}^{k}\|^2\right)\nn\\
&\ge& \frac{{c}}{3} \left(\|{\h x}^k-{\h x}^{k+1}\|
 +\|{\h y}^k-{\h y}^{k+1}\|
+\|{\h z}^{k+1}-{\h z}^{k}\|\right)^2\nn\\
&\ge& \frac{{c}}{3\zeta ^2}{\text{\rm dist}}(\h 0,\partial{\varpi} ({\h x}^{k+1},{\h y}^{k+1},{\h z}^{k+1},{\h z}^k))^2,\label{ekdes}\end{eqnarray}
where $c:=\min(c_1,c_2,c_3)>0$ and $\zeta>0$. Since $\varpi$ satisfies the KL property, by Theorem \ref{thm_conv} it follows that $({\h x}^k,{\h y}^k,{\h z}^k)_{k\geq 0}$ is convergent. We denote its limit by $(\overline{\h x},\overline{\h y},\overline{\h z})$, thus ${{\underline{\varpi}}}=\varpi(\overline{\h x},\overline{\h y},\overline{\h z},\overline{\h z})$. Let $\sigma>0$, $\eta>0$ and $\varepsilon>0$ be such that, if $\|({\h x},{\h y},{\h z},{\h u})-(\overline{\h x},\overline{\h y},\overline{\h z},\overline{\h z})\|<\varepsilon$ and $\underline{\varpi}<{{{\varpi}}}({\h x},{\h y},{\h z},{\h u})<\underline{\varpi} + \eta $, then it holds
$({\h x},{\h y},{\h z},{\h u}) \in \R^n \times \R^p \times \mathrm{int}(\dom g^*) \times \R^s$
and
 $$\sigma {\text{dist}}\Big({\bf 0},\partial {{{\varpi}}}({\h x},{\h y},{\h z},{\h u})\Big)\ge \Big({{{{\varpi}}}({\h x},{\h y},{\h z},{\h u})-\underline{\varpi}}\Big)^\theta.$$
Thus, there exists $k_0\ge K_0+2$ such that for all $k\ge k_0$
 $$\sigma {\text{dist}}\Big({\bf 0},\partial {{{\varpi}}}({\h x}^{k+1},{\h y}^{k+1},{\h z}^{k+1},{\h z}^{k})\Big)\ge \Big({{{{\varpi}}}({\h x}^{k+1},{\h y}^{k+1},{\h z}^{k+1},{\h z}^{k})-\underline{\varpi}}\Big)^\theta.$$
By combining this inequality with (\ref{ekdes}) it yields that for all $k \geq k_0$
$$ ({{\varpi}}_{k}-\underline{\varpi}) -({{\varpi}}_{k+1}-\underline{\varpi}) \ge \frac{c}{ 3\zeta ^2\sigma^2}({{\varpi}}_{k+1}-\underline{\varpi})^{2\theta}.$$
The conclusion follows directly from Lemma \ref{scalarrate}.
\end{proof}
The following lemma characterizes the convergence of the sequence $( {\h x}^k,{\h y}^k,{\h z}^k)_{k\ge 0}$ in terms of  the one of the sequence $({{\varpi}}_{k}-\underline{\varpi})_{k \geq 0}$.

\begin{lemma}\label{iteratecon}
Under the same assumptions as {Theorem \ref{thm_conv}}, let $({\h x}^k, {\h y}^k, {\h z}^k)_{k \geq 0}$ be the sequence generated by Algorithm \ref{AdaptiveDPFS}. Suppose that there exists $\theta \in [0,1)$ such that $\varpi$ satisfies the KL property at every $({\h x}, {\h y}, {\h z}, {\h u}) \in \dom \partial \varpi $ with desingularization function $\phi(s)=\frac{\sigma}{1-\theta}s^{1-\theta}$ for some $\sigma > 0$, and let $(\overline{\h x},\overline{\h y},\overline{\h z})$ be the limit of $({\h x}^k, {\h y}^k, {\h z}^k)_{k \geq 0}$ as $k \to +\infty$.
Then, there exist ${\widehat K}\ge K_0+2$ and a constant $\widehat C>0$ such that for all $k\ge {\widehat K}$ it holds
\begin{eqnarray*}
\|{\h x}^k-{\overline {\h x}}\| + \|{\h y}^k-{\overline {\h y}}\| + \|{\h z}^k-{\overline {\h y}}\| \le {\widehat C}\max\{\sqrt{{{\varpi}}_{k-1}-\underline{\varpi}},\;\phi({{\varpi}}_{k-1}-\underline{\varpi}) \}.
\end{eqnarray*}\end{lemma}
\begin{proof}
Let ${\widehat K}\ge K_0+2$ be the index introduced in the proof of Theorem \ref{thm_conv} with the property that for all $k \geq {\widehat K}$ the inequality (\ref{keydesc}) holds. 

Let $k \geq {\widehat K}$ and $n \geq k$. By summing (\ref{keydesc}) for $i=k,\ldots,n$, we obtain
\begin{eqnarray*}2S_{k}^n \le S_{k}^n +\|{\h x}^{k-1}-{\h x}^{k}\|+\|{\h y}^{k-1}-{\h y}^{k}\|+\|{\h z}^{k-1}-{\h z}^{k}\| +\frac{3\zeta }{c}\phi({{\varpi}}_{k}-\underline{\varpi}),\end{eqnarray*}
where $S_{k}^n=\sum_{i=k}^n \left(\|{\h x}^i-{\h x}^{i+1}\|+\|{\h y}^i-{\h y}^{i+1}\|+\|{\h z}^i-{\h z}^{i+1}\|\right)$.
Therefore, by using Theorem \ref{Algo1:conv} (i) and the notation ${c:}=\min(c_1,c_2,c_3)>0$,
\begin{align*}
S_{k}^n \le & \ \|{\h x}^{k-1}-{\h x}^{k}\|+\|{\h y}^{k-1}-{\h y}^{k}\|+\|{\h z}^{k-1}-{\h z}^{k}\| +\frac{3\zeta }{c}\phi({{\varpi}}_{k}-\underline{\varpi})\\
\le & \ \sqrt{3(\|{\h x}^{k-1}-{\h x}^{k}\|^2+\|{\h y}^{k-1}-{\h y}^{k}\|^2+\|{\h z}^{k-1}-{\h z}^{k}\|^2)} + \frac{3\zeta }{c}\phi({{\varpi}}_{k}-\underline{\varpi}) \\
\le & \ \frac{\sqrt{3}}{\sqrt{c}}\sqrt{{{\varpi}}_{k-1}-{{\varpi}}_{k}}+\frac{3\zeta }{c}\phi({{\varpi}}_{k}-\underline{\varpi}) \le \frac{\sqrt{3}}{\sqrt{c}}\sqrt{{{\varpi}}_{k-1}-{\underline{{\varpi}}}}+\frac{3\zeta }{c}\phi({{\varpi}}_{k-1}-\underline{\varpi}).
\end{align*}
The conclusion follows by choosing $\widehat C:=\max(\frac{\sqrt{3}}{\sqrt{c}},\frac{3\zeta }{c})$ and by letting $n \to +\infty$.
\end{proof}

By combining Theorem \ref{congrate} and Lemma \ref{iteratecon} we obtain the following convergence rate result.

\begin{theorem}\label{iterateconrate}
Under the same assumptions as {Theorem \ref{thm_conv}}, let $({\h x}^k, {\h y}^k, {\h z}^k)_{k \geq 0}$ be the sequence generated by Algorithm \ref{AdaptiveDPFS}. Suppose that there exists $\theta \in [0,1)$ such that $\varpi$ satisfies the KL property at every $({\h x}, {\h y}, {\h z}, {\h u}) \in \dom \partial \varpi $ with desingularization function $\phi(s)=\frac{\sigma}{1-\theta}s^{1-\theta}$ for some $\sigma > 0$, and let $(\overline{\h x},\overline{\h y},\overline{\h z})$ be the limit of $({\h x}^k, {\h y}^k, {\h z}^k)_{k \geq 0}$ as $k \to +\infty$. Then, the following statements are true:
\begin{itemize}
\item[(i)] If $\theta=0$, then $( {\h x}^k,{\h y}^k,{\h z}^k)_{k\ge 0}$ converges in finite time;
\item[(ii)] If $\theta\in(0,1/2]$, then there exist ${\widehat K}_1\ge K_0+2$ and ${\widehat q}\in(0,1)$, and a constant ${\widehat\kappa_1}>0$ such that for all $k\ge {\widehat K}_1$
\begin{eqnarray*}
\|{\h x}^k-{\overline {\h x}}\|\le {\widehat\kappa_1}{\widehat q}^k, \quad
\|{\h y}^k-{\overline {\h y}}\|\le {\widehat\kappa_1}{\widehat q}^k, \quad
\|{\h z}^k-{\overline {\h z}}\|\le {\widehat\kappa_1}{\widehat q}^k;
\end{eqnarray*}
\item[(iii)] If $\theta\in(1/2,1)$, then there exist ${\widehat K}_1\ge K_0+2$  and a constant ${\widehat\kappa_2}>0$ such that for all $k\ge {\widehat K}_1$
\begin{eqnarray*}
\|{\h x}^k-{\overline {\h x}}\|\le {\widehat\kappa_2}k^{-\frac{1-\theta}{2\theta-1}}, \quad
\|{\h y}^k-{\overline {\h y}}\|\le {\widehat\kappa_2}k^{-\frac{1-\theta}{2\theta-1}}, \quad
\|{\h z}^k-{\overline {\h z}}\|\le {\widehat\kappa_2}k^{-\frac{1-\theta}{2\theta-1}}.
\end{eqnarray*}
\end{itemize}
\end{theorem}


\section{On the KL property of the merit function}\label{sec7}

In this section, we provide conditions in terms of the functions that appear in the formulation of the optimization problem \eqref{PForm} that guarantee that the merit function satisfies the KL property. To proceed, we denote by
$${\tau}:={\alpha} \left(1+\frac{4\alpha}{{\widetilde\alpha}-\alpha}\right) \quad \mbox{and} \quad {\widetilde{\gamma}}:={\widetilde \alpha}-\alpha,$$
and therefore the merit function reads
\begin{eqnarray}\label{varpi}
&&{\varpi}({\h x},{\h y},{\h z},{\h u})=\Psi({\h x},{\h y}, {\h z})
 +\frac{\tau}{2}\|{\h z}-{\h u}\|^2-\frac{\widetilde{\gamma}}{2}\|{\h z}\|^2.
\end{eqnarray}
The following lemma from \cite{LiPong18} will be important in our analysis.
\begin{lemma}\label{lem:equiv_norms}
Let $t > 0$. Then there exist $C_1 \geq C_2 > 0$ such that
\[
C_2 \|{\h a}\| \leq \left( a_1^t + \cdots + a_m^t \right)^{\frac{1}{t}} \leq C_1 \|{\h a}\| \quad \forall {\h a} = (a_1, \ldots, a_m) \in \mathbb{R}_+^m.
\]
\end{lemma}
Considering also Remark \ref{rem:local_lsc}, $\varpi$ satisfies the KL property at any noncritical point (see for instance {\cite{AB}}). Therefore, in the following, we will only focus on statements concerning critical points.

First, we provide conditions to ensure that the KL property of $F+\iota_S$ can be passed directly to $\Psi$.

\begin{theorem}
\label{KLexp}
Let $\overline{\h x}$ be a critical point of the optimization problem \eqref{PForm}. Assume that
\begin{enumerate}
    \item[(i)] $g$ is essentially strictly convex;
    \item[(ii)] $g$ is differentiable and has a Lipschitz continuous gradient on the open ball ${B}({A}\overline{\h x},\delta_1)$ for some $\delta_1>0$;
    \item[(iii)] $f$ is differentiable and has a $\ell_{\nabla f}$-Lipschitz continuous gradient on the open ball $B(K\overline{\mathbf{x}},\delta_2)$ for some $\delta_2>0$;
    \item[(iv)]  $f^*$ is continuously differentiable on the open ball ${ B}(\nabla f (K\overline{\mathbf{x}}),\delta_3)$ for some $\delta_3 >0$.
\end{enumerate}
Then, there exists $({\overline{\h y}},{\overline{\h z}}) \in \R^p \times \R^s$ such that
\begin{eqnarray} \label{eq:unique-triple}
{\overline {\h y}} = \nabla f( K{\overline{\h x}}), \quad {\overline {\h z}} = \nabla g( A{\overline{\h x}}) \quad \mbox{and} \quad
K^* {\overline{\h y}} \in A^* {\overline{\h z}}+\nabla\varphi({\overline{\h x}}) + N_{S}(\overline{\h x}).\end{eqnarray}
If, in addition, $F+ \iota_S$  satisfies the KL property at $\overline{\h x}$ with desingularization function $\phi(s)=\frac{\sigma_F}{1-\theta}s^{1-\theta}$ for some $\sigma_F >0$ and $\theta \in [\frac{1}{2},1)$, then, the function $\Psi$ defined in (\ref{Phif}) satisfies the KL property at $(\overline{\h x}, \overline{\h y}, \overline{\h z})$ with desingularization function $\phi(s)=\frac{\sigma_\Psi}{1-\theta}s^{1-\theta}$ for some $\sigma_\Psi>0$.
\end{theorem}
\begin{proof} The KKT optimality condition \eqref{eq:unique-triple} is a consequence of the differentiability assumptions (ii) and (iii).

By assumption, there exist $\varrho,\eta >0$, such that the KL inequality
 \begin{eqnarray*}F({\h x})- F({\overline{\h x}}) \leq \sigma_F^\frac{1}{\theta} {\text{dist}}^{\frac{1}{\theta}}({\bf 0},\partial (F + \iota_S)({\h x}))\end{eqnarray*}
 holds whenever $\h x \in S$, $\|{\h x}-{\overline{\h x}}\|< \varrho$ and $F({{\h x}}) < F({\overline{\h x}}) + \eta $. Here, the condition $F({\overline{\h x}})< F({{\h x}})$ was omitted, since the inequality holds trivially otherwise.

Let $\epsilon_0>0$ such that ${{\h x}\in B(\overline{\h x},\epsilon_0)}$ implies $A{\h x} \in {B}(A\overline{\h x},\delta_1)$, $K{\h x} \in {B}(K\overline{\h x},\delta_2)$, $\norm{K {\h x} - K \overline{\h x}}<\frac{1}{2}$ and $   F({{\h x}})  < F({\overline{\h x}}) + \eta $.  This is possible, since our assumptions ensure that $F$ is continuous around $\overline{\h x}$.

By assumption (i), we have $g^*$ is differentiable on $\dom \partial g^* = \mathrm{int}(\dom g^*)$. Since $A\overline{\h x} \in \partial g^*(\overline{\h z} )$, it follows that $\overline{\h z}\in \mathrm{int}(\dom g^*)$. Hence, we can choose $\epsilon_1>0$ such that $\norm{\h z - \overline{\h z}}<\epsilon_1$ implies that ${\h z } \in \mathrm{int}(\dom g^*)$, so that $g^*$ is differentiable at ${\h z}$.

Since $\nabla f^*$ is continuous at $\overline{\h y} = \nabla f( K{\overline{\h x}})$, there exists $\epsilon_2>0$ such that
\begin{align}\label{est:nablaf*cont}
    \norm{\nabla f^*(\h y) - K\overline{\h x}}=\norm{\nabla f^*({\h y}) - \nabla f^*(\overline{\h y})} < \min(\tfrac{1}{2},\delta_2)
\end{align}
whenever $\norm{\h y - \overline{\h y}}<\epsilon_2$.

 Since $g$ is differentiable and has a Lipschitz continuous gradient on a neighborhood of $A\overline{\h x}$, we have by \cite[Theorems 3.6, 3.10]{artacho_metric_reg} that $\nabla g^*$ is metrically regular at $(\overline{\h z},A{\overline{\h x}})$. This means that there exist
   $\ell_1>0$ and $\epsilon_3>0$ such that
                       \begin{eqnarray}\label{est:g_metric}\|\nabla g(A{\h x})-{\h z}\|\le \ell_{1} \norm{A{\h x} - \nabla g^*({\h z)}}
                       \end{eqnarray}
                       whenever $\|{\h z}-{\overline{\h z}}\|<\min (\epsilon_1,\epsilon_3)$ and $\|{\h x}-{\overline{\h x}}\|<\epsilon_3$.

Analogously, $\nabla f^*$ is metrically regular at $(\overline{\h y},K{\overline{\h x}})$, thus there exist
   $\ell_2>0$ and $\epsilon_4>0$ such that
                        \begin{eqnarray}\label{est:f_metric}\|\nabla f(K{\h x})-{\h y}\|\le { }\ell_{2} \norm{{K{\h x}- \nabla f^*({\h y})}}\end{eqnarray}
                        whenever $\|{\h y}-{\overline{\h y}}\|<\epsilon_4$ and $\|{\h x}-{\overline{\h x}}\|<\epsilon_4$.

Now choose $({\h x}, {\h y }, {\h z})$ such that
\begin{align*}
    \norm{{\h x} - \overline{\h x}} &< \min (\varrho, \epsilon_0, \epsilon_3, \epsilon_4), \\
    \norm{{\h y} - \overline{\h y}} &< \min (\epsilon_2,\epsilon_4), \\
    \norm{{\h z} - \overline{\h z}}  &< \min (\epsilon_1,\epsilon_3),
\end{align*}
and $\Psi(\overline{\h x},\overline{\h y},\overline{\h z})< \Psi({\h x}, {\h y}, {\h z}) < \Psi(\overline{\h x},\overline{\h y},\overline{\h z}) + 1$ hold. Then ${\h x}\in S$ must hold.
 We now prove that there exists $\sigma_\Psi>0$ such that $\Psi$ satisfies the KL property at $(\overline{\h x}, \overline{\h y }, \overline{\h z})$ with desingularization function $\phi(s)=\frac{\sigma_\Psi}{1-\theta}s^{1-\theta}$.

 Since ${\overline{\h y}}=\nabla f(K{\overline{\h x}})$ and ${\overline{\h z}}= \nabla g(A{\overline{\h x}})$, it holds $\Psi({\overline{\h x}},{\overline{\h y}},{\overline{\h z}})=F({\overline{\h x}})$. Therefore,
\begin{eqnarray*}&&\Psi({\h x},{\h y},{\h z})-\Psi({\overline{\h x}},{\overline{\h y}},{\overline{\h z}})\nn\\
= && \ \langle{\h z}, A{\h x} \rangle-g^*({\h z}) + \varphi({\h x}) -\langle K{\h x},{\h y}\rangle+f^*({\h y})-F({\overline{\h x}})\nn\\
= && \ \underbrace{\langle{\h z}, A{\h x} \rangle-g^*({\h z})-g(A{{\h x}})}_{\le 0}+F({\h x})-F({\overline{\h x}})
+\underbrace{f(K{\h x})+f^*({\h y})-\langle K{\h x},{\h y} \rangle}_{\diamond}.\end{eqnarray*}
By the local Lipschitz continuity of the gradient of $f$, we obtain
\begin{eqnarray*}\diamond&=& f(K{\h x}) + \langle \nabla  f^*({\h y}),{\h y} \rangle-f(\nabla  f^*({\h y}))- \langle K{\h x},{\h y} \rangle\nn\\
                        &=& f(K{\h x}) -f(\nabla  f^*({\h y})) +  \langle \nabla  f^*({\h y}) -K{\h x},{\h y} \rangle \nn\\
                        & \leq & \frac{\ell_{\nabla f}}{2}\norm{\nabla  f^*({\h y})-K{\h x}}^{2} \nn \\
                        &\leq & \frac{\ell_{\nabla f}}{2}\norm{{\nabla f^*({\h y})}-K{\h x}}^{\frac{1}{\theta}},
                        \end{eqnarray*}
                        since $\nabla  f^*({\h y}), K{\h x}\in {B}(K\overline{\h x},\delta_2)$, $\norm{\nabla  f^*({\h y})-K{\h x}} \leq \norm{\nabla  f^*({\h y})-K\overline{\h x}} + \norm{ K{\h x}-K\overline{\h x}} < 1$  and $\theta \geq  \frac{1}{2} $.
                        On the other hand, we have
                        \begin{eqnarray*} {\text{dist}}({\bf 0},\partial (F+\iota_S)({\h x}))&\le&  \norm{ A^* \nabla g(A{\h x})-A^* {\h z}}\nn\\
 &&+{\text{ dist}}\Big ({\bf 0},A^*{\h z}+\nabla \varphi({\h x})+\partial\iota_{S}(\h x) -K^* {\h y} \Big ) +\partial\iota_{S}(\h x) -K^* {\h y}) \nn \\
                       & &+\norm{K^*{\h y} -K^* \nabla f(K{\h x})}.\end{eqnarray*}
According to Lemma \ref{lem:equiv_norms}, there exists $c_2>0$ such that
 \begin{eqnarray*} {\text{dist}}^{\frac{1}{\theta}}({\bf 0},\partial (F+\iota_S)({\h x}))&\le&c_2\big( \sigma_A^{\frac{1}{2\theta}}\norm{ \nabla g(A{\h x})- {\h z}}^{\frac{1}{\theta}}\nn\\
 &&+{\text{ dist}}^{\frac{1}{\theta}} \Big ({\bf 0},A^*{\h z}+\nabla \varphi({\h x})+\partial\iota_{S}(\h x) -K^* {\h y} \Big ) \nn \\
    & &+\sigma_K^{\frac{1}{2\theta}}\norm{{\h y} -  \nabla f(K{\h x})}^{\frac{1}{\theta}}\big),\end{eqnarray*}
                       where $\sigma_A=\norm{A}^2$ and $\sigma_K=\norm{K}^2$.
Combining this with the KL inequality for $F$, \eqref{est:nablaf*cont}, \eqref{est:g_metric} and \eqref{est:f_metric}, we obtain
                        \begin{eqnarray*} &&\Psi({\h x},{\h y},{\h z})-\Psi({\overline{\h x}},{\overline{\h y}},{\overline{\h z}})\nn\\
                        \le && \ c_2{\sigma_F^{\frac{1}{\theta}}}\bigg[\sigma_A^{\frac{1}{2\theta}}\ell_{1}^{\frac{1}{\theta}} \norm{A{\h x} - \nabla g^*({\h z})}^{\frac{1}{\theta}}+{\text{dist}}^{\frac{1}{\theta}}\Big ({\bf 0},A^*{\h z}+\nabla \varphi({\h x})+\partial\iota_{S}(\h x) -K^* {\h y} \Big )\nn\\
                        &&\quad \quad \quad +\sigma_K^{\frac{1}{2\theta}}\ell_{2}^{\frac{1}{\theta}} \norm{K{\h x} - \nabla f^*(\h y)}^{\frac{1}{\theta}}\bigg] + \frac{\ell_{\nabla f}}{2}\norm{K{\h x}- \nabla f^*({\h y})}^\frac{1}{\theta}\nn\\
                        \le && \ \sigma_{\Psi}^\frac{1}{\theta} {\text{dist}}{^{\frac{1}{\theta}}}\left({\bf 0}, \left(\begin{array}{c} A^*{\h z}+\nabla \varphi({\h x})+\partial\iota_{S}(\h x) -K^* {\h y}\\
                        -K{\h x}+\nabla f^*({\h y})\nn\\
                        A{\h x}-\nabla g^*({\h z})\end{array}\right)\right),\end{eqnarray*}
                        where $\sigma_{\Psi}:=C_1 {\max\left(
c_2\sigma_F^{1/\theta}\sigma_A^{1/2\theta}\ell_1^{1/\theta}, c_2 \sigma_F^{1/\theta},c_2 \sigma_F^{1/\theta}\sigma_K^{1/2\theta}\ell_2^{1/\theta}+\frac{1}{2}\ell_{\nabla f} \right)^{\theta}}
$ and $C_1>0$ is the constant provided by Lemma \ref{lem:equiv_norms}.

To summarize, there exists $\varrho_\Psi>0$ and $\sigma_\Psi>0$ such that for $({\h x},{\h y}, {\h z})$ with $\norm{({\h x},{\h y}, {\h z}) - (\overline{\h x},\overline{\h y},\overline{\h z}) } < \varrho_\Psi$ and $\Psi({\overline{\h x}},{\overline{\h y}},{\overline{\h z}}) < \Psi({\h x},{\h y},{\h z}) < \Psi({\overline{\h x}},{\overline{\h y}},{\overline{\h z}}) + 1$ it holds
                        \begin{align*}
                            (\Psi({\h x},{\h y},{\h z})-\Psi({\overline{\h x}},{\overline{\h y}},{\overline{\h z}}))^\theta \leq \sigma_\Psi \dist({\bm  0, \partial \Psi({\h x},{\h y},{\h z})}).
                        \end{align*}

\end{proof}

The following result is an extension of \cite[Theorem 3.6]{LiPong18}.
\begin{lemma}\label{lem:LiPong}
     Let $h:{\R^{N_1}\times \dots \times \R^{N_n}} \to \overline{\R}$ be a proper and lower semicontinuous function and let $(\overline{\h x}_1, \dots, \overline{\h x}_n)$ be a critical point of $h$.  Suppose that $h$ satisfies the KL property at  $(\overline{\h x}_1, \dots, \overline{\h x}_n)$  with desingularization function $\phi(s)=\frac{\sigma_h}{1-\theta}s^{1-\theta}$ for some $\sigma_h > 0$ and $\theta \in [\frac{1}{2},1)$. For $i \in \{1,\dots,n\}$ and $\beta >0$, consider the function $H_i : \R^{N_1}\times \dots \times \R^{N_n} \times \R^{N_i} \to \overline{\R}, H_i(\h x_1,\dots,\h x_n,\h y):=h(\h x_1,\dots,\h x_n)+\frac{\beta}{2}\norm{\h x_i-\h y}^2 $. Then $H_i$ satisfies the KL property at its critical point $(\overline{\h x}_1, \dots, \overline{\h x}_n, \overline{\h x}_i)$ with desingularization function $\phi(s)=\frac{\sigma_H}{1-\theta}s^{1-\theta}$ for some $\sigma_H>0$.
     \begin{proof}
 For simplicity, we denote ${\h x} =({\h x}_1, ..., {\h x}_n)$ and $\overline{\h x} =(\overline{\h x}_1, ..., \overline{\h x}_n)$. By assumption, there exist $\sigma_h>0$, $\rho>0$ and $\eta \in (0,+\infty]$ such that
        \begin{align*}
            \sigma_h^\frac{1}{\theta} \dist^{\frac{1}{\theta}}(0,\partial h(\h x)) \geq  h(\h x)-h(\overline{\h x})
        \end{align*}
        whenever $\norm{\h x-\overline{ \h x}}<\rho$ and $ h(\h x) < h(\overline{\h  x} )+\eta$. Now consider $(\h x,\h y)$ with $\h x\in \dom \partial h$, $\norm{\h x-\overline{\h x}}\leq \min(\rho,\frac{1}{2})$, $ \norm{\overline{\h x}_i - \h y}<\min(\frac{1}{2},\sqrt{\frac{\eta}{\beta}})$ and $H_i(\overline{ \h x}, \overline{ \h x}_i) < H_i(\h x,\h y) < H_i(\overline{\h  x}, \overline{\h  x}_i)+\frac{\eta}{2}$. Then $h(\h x) < h(\overline{\h x}) + \eta$ and
        \begin{align*}
            \dist^{\frac{1}{\theta}}(0,{\partial}H_i(\h x,\h y))&= \!\! \! \inf_{{\substack{\xi \in \partial h({\h x}),\\ \xi:=(\xi_1,\ldots,\xi_n)}} } \left(\norm{\beta(\h x_i-\h y)}^2+\norm{\beta(\h x_i-\h y)+\xi_i}^2+\sum_{\substack{
                 j=1, j\neq i}}^n\norm{\xi_j}^2\right)^{\frac{1}{2\theta}} \\
            &\geq{\tilde C_1} \inf_{\xi \in \partial h(x)}\left( \norm{\beta({\h x_i}-\h y)}^{\frac{1}{\theta}}+\norm{\beta(\h x_i-\h y)+\xi_i}^{\frac{1}{\theta}}+\sum_{j=1,j\neq i}^n\norm{\xi_j}^{\frac{1}{\theta}}\right),
        \end{align*}
where ${\tilde C_1}>0$ {is the constant provided by  Lemma \ref{lem:equiv_norms}}. Furthermore, we can apply \cite[Lemma 3.1]{LiPong18} with $\nu_1>0$ and $\nu_2\in(0,1)$, and get
        \begin{align*}
        \dist^{\frac{1}{\theta}}(0, {\partial}H_i(\h x,\h y))&\geq {{\tilde C_1}} \inf_{\xi \in \partial h(x)} \left((1-\nu_2)\norm{\beta(\h x_i-\h y)}^{\frac{1}{\theta}}+\nu_1\norm{\xi_i}^{\frac{1}{\theta}}+\sum_{j=1,j\neq i}^n\norm{\xi_j}^{\frac{1}{\theta}}\right)\nn\\
       & \geq  {{\tilde C_1}} \inf_{\xi \in \partial h(x)} \left((1-\nu_2)\norm{\beta(\h x_i-\h y)}^{\frac{1}{\theta}}+{\min(\nu_1,1)(\norm{\xi_i}^{\frac{1}{\theta}}+\sum_{j=1,j\neq i}^n\norm{\xi_j}^{\frac{1}{\theta}})}\right)\nn\\
        &\geq {{\tilde C_1}} \inf_{\xi \in \partial h(x)} \left((1-\nu_2)\norm{\beta(\h x_i-\h y)}^{\frac{1}{\theta}}+{\min(\nu_1,1){ C}_2^{1/\theta}\dist^{\frac{1}{\theta}}(0,\partial h(\h x))}\right),
        \end{align*}
where ${C_2}>0$ {is the constant provided by  Lemma \ref{lem:equiv_norms}}. This yields, for $\tilde  C_2 := {\tilde C_1}\min((1-\nu_2), {\min(\nu,1){C}_2^{1/\theta}})$,
        \begin{align*}
        \dist^{\frac{1}{\theta}}(0, {\partial}H_i(\h x,\h y))\geq  \tilde C_2 \left(\beta^{\frac{1}{\theta}}\norm{\h x_i-\h y}^{\frac{1}{\theta}} + \dist^{\frac{1}{\theta}}(0,\partial h(\h x))  \right).
        \end{align*}
        The last inequality is due to Lemma \ref{lem:equiv_norms}.
Using that $\norm{{\h x}_i - {\h y}}\leq \norm{{\h x}_i-{\overline{\h x}_i}} + \norm{\overline{\h x}_i- {\h y}}<1$, the KL property for $h$ yields
        \begin{align*}
        \dist^{\frac{1}{\theta}}(0, {\partial}{H_i}(\h x,\h y))\geq \tilde C_2 \left(\sigma_h^{-\frac{1}{\theta}}(h(\h x)-h(\overline{ \h x})) + \beta^{\frac{1}{\theta}} \norm{\h x_i-\h y}^2  \right).
        \end{align*}
This allows us to conclude, for $C_3 := \tilde C_2\min ( \sigma_h^{-\frac{1}{\theta}}, 2\beta^{\frac{1}{\theta}-1})$, that
        \begin{align*}
        \dist^{\frac{1}{\theta}}(0, {\partial}H_i(\h x,\h y)) \geq  C_3(H_i(\h x,\h y)-H_i(\overline{\h  x} , \overline{\h x}_i )).
        \end{align*}
     \end{proof}
\end{lemma}

The following theorem introduces a setting which guarantees that $\varpi$ has the KL property with desingularization function $\phi(s)=\frac{\sigma}{1-\theta}s^{1-\theta}$, which by Theorem \ref{iterateconrate} is sufficient to prove a linear convergence rate of the generated sequence.

\begin{theorem}\label{KLexana}
For ${\widetilde{\gamma}} = \widetilde{\alpha} - \alpha$, let $F_{{\widetilde{\gamma}}}(\h x):= \mathcal{G}_{g,\widetilde{\gamma}}(A{\h x})+\varphi(\h x)-f(K{\h x})$ and $\overline{\h x}$ be a critical point of $F_{{\widetilde{\gamma}}}+\iota_S$ with the property that $F_{{\widetilde{\gamma}}}+\iota_S$ satisfies the KL property at  $\overline{\h x}$  with desingularization function $\phi(s)=\frac{\sigma}{1-\theta}s^{1-\theta}$ for some $\sigma >0 $ and $\theta \in [\frac{1}{2},1)$.
Assume that $g$ is essentially strictly convex and $f$ satisfies the assumptions (iii)-(iv) of Theorem \ref{KLexp} for $\overline{\h x}$.
Then, there exists $({\overline{\h y}},{\overline{\h z}}) \in \R^p \times \R^s$ such that
\begin{eqnarray}\label{KKT-KL}
{\overline {\h y}} = \nabla f( K{\overline{\h x}}), \quad {\overline {\h z}} = \nabla \mathcal{G}_{g,\widetilde{\gamma}}( A{\overline{\h x}}) \quad \mbox{and} \quad
K^* {\overline{\h y}} \in A^* {\overline{\h z}}+\nabla\varphi({\overline{\h x}}) + N_{S}(\overline{\h x}),\end{eqnarray}
thus $(\overline{\h x},\overline{\h y},\overline{\h z},\overline{\h z})$ is a critical point of $\varpi$.
In addition, $\varpi$ satisfies the KL property at $(\overline{\h x}, \overline{\h y}, \overline{\h z}, \overline{\h z})$  with desingularization function $\phi(s)=\frac{\sigma'}{1-\theta}s^{1-\theta}$ for some $\sigma'>0$.
\end{theorem}
\begin{proof}
The Moreau envelope $\mathcal{G}_{g,\widetilde{\gamma}}$ is differentiable and it has a Lipschitz continuous gradient everywhere, and it inherits the essentially strict convexity from $g$ (\cite{BC17}). The existence of $\overline{\h y}$ and $\overline{\h z}$ is a consequence of Theorem \ref{KLexp}. Since
\begin{align*}
     A{\overline{\h x}} \in \partial \mathcal{G}_{g,\widetilde{\gamma}}^*(\overline{\h z}) = \partial g^*(\overline{\h z})+ {\widetilde{\gamma}}A{\overline{\h z}},
\end{align*}
thus
 $(\overline{\h x},\overline{\h y},\overline{\h z},\overline{\h z})$ is a critical point of $\varpi$.
The last statement follows from Theorem \ref{KLexp} and {Lemma} \ref{lem:LiPong}, by also considering Remark \ref{rem:local_lsc}.
\end{proof}

\section{Application to audio signal processing} \label{sec:Application}

Consider an audio signal of duration $T$ seconds, sampled at a rate of $f_s$ samples per second. For simplicity, we consider a mono audio signal, which we can represent by a vector $\h x\in \R^N $ of size $N=T \cdot f_s$, whose entries are typically in the range of $[-1,1]$, where $-1$ represents the minimum amplitude, and 1 represents the maximum amplitude of the audio signal.
The original audio signal $\overline{\h x}$ is assumed to be corrupted with a noise $\eta $ such that only the resultant $\h u=\overline{ \h x}+\eta$ is known to us.

We aim to reconstruct the original audio signal $\overline{\h  x}$ by considering a minimization problem that aims to find and estimate ${ \h x}$ that is as close as possible to the original $\overline{ \h x}$. This typically involves formulating an objective function that consists of a fidelity term $\norm{\h x-\h u}_2^2$ and a regularization term that imposes specific priors or assumptions about the signal $\h x$.

Audio signals usually are not sparse in their time domain, but their time-frequency representation obtained through the Short-Time Fourier Transform (STFT) or discrete Gabor Transform can exhibit sparsity. The STFT is defined as an injective linear operator $ \mathcal{T} :\R^N \to \C^{M}$, which maps a signal into the time-frequency domain \cite{Feichtinger98}.
We want to reconstruct the original signal $\h x$ by considering the minimization problem
\begin{align*}
    \min_{\h x\in \R^N} \frac{r_1}{2}\norm{\h x-\h u}_2^2+R(\mathcal{T} \h x),
\end{align*}
where we denote by $\norm{\cdot}_2$ the usual Euclidean norm, $r_1>0$ is a regularization parameter and $R:\C^M\to \R$ is a regularizing functional promoting sparsity.

\subsection{Difference of Convex (DC) regularizer}

It is widely recognized that $\|\cdot\|_1$ minimization encourages sparsity. Thanks to this property—along with its convexity—it has become a popular tool for promoting sparse solutions in optimization problems, especially in areas like compressed sensing and sparse signal recovery. However, despite these advantages, convex regularizers have inherent limitations when it comes to capturing more intricate behaviors, which can lead to suboptimal performance in certain cases \cite{Example}.

Recently, more sophisticated sparsity-promoting regularizers have been explored, such as $\|\cdot\|_1 - \|\cdot\|_2$ and $\frac{\|\cdot\|_1}{\|\cdot\|_2}$. These regularizers effectively capture sparse behavior, as demonstrated in \cite{Yin14}. Nonetheless, their adoption has been constrained by the lack of fast and efficient algorithms.

In this section we will consider regularizers of the form $\|\cdot\|_1 - {r_2} \|\cdot\|_2$ ($0 < r_2 <1$). A simple interpretation of its sparsity-inducing properties is as follows: the sparser the vector, the closer its $\|\cdot\|_1$ norm and $\|\cdot\|_2$ norm are. This difference highlights the balance between the sparsity promotion of the $\|\cdot\|_1$ norm and the error distribution balancing of the $\|\cdot\|_2$ norm, making it a potentially powerful tool in optimization problems where both properties are desirable. Additionally, the parameter ${r_2}$ provides further flexibility for various applications.

We will therefore consider the minimization problem
\begin{align} \label{eq:num_KL_objective}
    \min_{\h x\in S} \varphi(\h x)+g(\mathcal{T} \h x)-f(\mathcal{T} \h x),
\end{align}
where
\begin{align}
    &g(\h x) := \norm{\h x}_1+\frac{\lambda}{2}\norm{\h x}_2^2, \notag \\
    &f(\h x) := r_2 \norm{\h x}_2+\frac{\lambda}{2}\norm{\h x}_2^2, \\
     &\varphi(\h x):= \frac{r_1}{2}\norm{\h x- {\h u}}_2^2, \notag \\
    &S:= \{\h x \in \R^N:  \norm{{\h x} - {\h u} }_2 \leq p \norm{\h u}_2\}, \label{eq:num_fgphi_objective}
\end{align}
$\h u \neq 0$ is a fixed vector, corresponding to the corrupted measurement, $r_1$, $r_2$, ${\lambda} >0$ are parameters, and $\mathcal{T}:\R^N\to \R^{2M}$ is an injective linear operator. Here we use a simple zero sum trick to make $g$ and $f$ strongly convex, while keeping $g(\mathcal{T}\h x)-f(\mathcal{T}\h x) = \norm{\mathcal{T}\h x}_1 - r_2 \norm{\mathcal{T}\h x}_2$ as desired. The constant $p\in (0,1)$ serves as an upper bound on the relative noise level in $\h u$, i.e., $\frac{\norm{\eta}_2}{\norm{\h u}_2}$. In the absence of \emph{a priori} knowledge, one can reasonably choose $p\leq \frac{1}{2}$ which corresponds to the scenario where $\norm{\eta}_2\leq \norm{\overline{\h x}}_2$. Recovering signals with higher noise levels is generally infeasible unless additional assumptions are imposed.

Note that $f$ and $g$ are strongly convex with modulus $\lambda$, and $\varphi$ is strongly convex with modulus $r_1$. Moreover, $f$ is differentiable with a locally Lipschitz continuous gradient over any compact convex set not containing the origin. Thus the conditions on $f$ and $g$ in Theorem \ref{KLexp} are met at every ${\h x}\in S$.

\subsection{KL exponent of the merit function}

For ${\widetilde{\gamma}} = \widetilde{\alpha} - \alpha$, let \begin{align}\label{fnu_bar}
{ F}_{{\widetilde{\gamma}}}(\h x)&:=  \mathcal{G}_{g,\widetilde{\gamma}}( {\cal T}{\h x})+\varphi({\h x})-{ f}( {\cal T}{\h x}), 
\end{align} 
We want to prove that ${ F}_{\widetilde{\gamma}} + \iota_S$ satisfies the KL property at every critical point $\overline{\h x}$  with desingularization function $\phi(s)=\frac{\sigma}{1-\theta}s^{1-\theta}$ with $\theta=\frac{1}{2}$ for some $\sigma >0 $. Then, by Theorem \ref{KLexana}, the corresponding merit function ${\varpi}$ defined in (\ref{varfun}) will satisfy the KL property at every critical point.

Indeed, for every critical point $(\overline{\h x}, \overline{\h y}, \overline{\h z}, \overline{\h u})$ of $\varpi$, it holds $\overline{\h u} = \overline{\h z}$, and $(\overline{\h x}, \overline{\h y}, \overline{\h z})$ satisfies \eqref{KKT-KL}. This means that $\overline{\h x}$ is a critical point of $F_{\widetilde{\gamma}}+\iota_S$, therefore, by Theorem \ref{KLexana}, $\varpi$ satisfies the KL property at $(\overline{\h x}, \overline{\h y}, \overline{\h z}, \overline{\h u})$ with desingularization function $\phi(s) = 2\sigma's^\frac{1}{2}$ for some $\sigma'>0$.

\begin{lemma}\label{F_str_conv}Let ${\cal T}$ be an injective linear operator and $\mu>0$ be such that $\| {\cal T}\h{x} \|_2 \ge \mu$ for every $\h{x} \in S$. Then ${F}_{\widetilde{\gamma}}$ is differentiable on $S$ and it has a Lipschitz continuous gradient.
Moreover, if
$$
r_1 > \sigma_\mathcal{T}\left(\frac{2r_2}{\mu}+\lambda\right),
$$
then ${F}_{\widetilde{\gamma}}$ is strongly convex on $S$.
\end{lemma}
\begin{proof} Note that ${\bm 0}\notin S$, therefore ${F}_{\widetilde{\gamma}}$ is differentiable on $S$. For every ${\h x}, {\h y}\in S$ it holds
\begin{align*}
    \norm{\frac{\h x}{\norm{\h x}_2}-\frac{\h y}{\norm{\h y}_2}}_2 \leq \frac{2}{\min_{\h x\in S} \norm{\h x}_2} \norm{\h x - \h y}_2.
\end{align*}
Thus $f\circ \mathcal{T}$ has a $\sigma_\mathcal{T}(\frac{2r_2}{\mu}+\lambda)$-Lipschitz continuous gradient on $S$. Since $\mathcal{G}_{g,\widetilde{\gamma}}$ and $\varphi$ have Lipschitz continuous gradient everywhere, ${F}_{\widetilde{\gamma}}$ has Lipschitz continuous gradient on $S$.

Furthermore $-f\circ \mathcal{T}$ is weakly convex on $S$ with modulus $\sigma_\mathcal{T}(\frac{2r_2}{\mu}+\lambda)$. 
Since $ \mathcal{G}_{g,\widetilde{\gamma}}\circ \mathcal{T}$ is convex and $\varphi$ is strongly convex with modulus $r_1$, it yields that ${F}_{\widetilde{\gamma}}$ is strongly convex on $S$ if $r_1 > \sigma_\mathcal{T}(\frac{2r_2}{\mu}+\lambda)$.
\end{proof}

One could also leverage the strong convexity of $ \mathcal{G}_{g,\widetilde{\gamma}}\circ \mathcal{T}$ to relax the condition on $r_1$. However, its strong convexity modulus depends on $\widetilde{\gamma}=\widetilde{\alpha}-\alpha$, a quantity that varies with the iterates of the algorithm.

\begin{theorem}\label{KL1o2}
Let ${\cal T}$ be an injective linear operator and $\mu>0$ be such that $\norm{{\cal T}\h{x}}_2 \ge \mu$ for all $\h{x} \in S$. If $r_1 > \sigma_\mathcal{T}(\frac{2r_2}{\mu}+\lambda)$, then ${F}_{\widetilde{\gamma}}+\iota_S$ satisfies the KL property at every critical point $\h{x}$ with desingularization function $\phi(s) = 2\sigma s^{\frac{1}{2}}$ for some $\sigma >0$.
\end{theorem}

\begin{proof}
Let $\overline{\h{x}}$ be a critical point of $F_{\widetilde{\gamma}}+\iota_S$. We distinguish two cases:

\medskip
\noindent \textbf{Case 1:} $\overline{\h{x}} \in \mathrm{int}(S)$.

By the strong convexity of ${F}_{\widetilde{\gamma}}$ on $S$, there exists $\widetilde{\mu} > 0$ and a neighborhood $U \subseteq \mathrm{int} (S)$ of $\overline{\h{x}}$ such that for every $\h{x} \in U$
\begin{align*}
\frac{1}{2\widetilde{\mu}} \dist^2(\bm 0, \partial (F_{\widetilde{\gamma}}+\iota_S)(\h{x})) 
= \frac{1}{2\widetilde{\mu}} \norm{\nabla {F}_{\widetilde{\gamma}}(\h{x})}^2 
&\ge {F}_{\widetilde{\gamma}}(\h{x}) - {F}_{\widetilde{\gamma}}(\overline{\h{x}}) 
\end{align*}
Thus, the KL property holds in this case.

\medskip
\noindent
\textbf{Case 2:} $\overline{\h{x}} \in \mathrm{bd}(S)$.

By the structure of $S$, the normal cone of $S$ at ${\h x }\in \mathrm{bd}(S)$ is
$$
N_S({\h{x}}) = \{ \delta ({\h{x}} - \h{u}) \mid \delta \ge 0 \}. 
$$
Let $\overline{\delta} \ge 0$ be such that $\overline{\xi} = \overline{\delta}  (\overline{\h{x}} - \h{u}) \in N_{S}(\overline{\h x})$ satisfies
$$
\nabla {F}_{\widetilde{\gamma}}(\overline{\h{x}}) + \overline{\xi} = 0.
$$
Let $\h{x} \in S$ and every $\xi \in N_S(\h{x})$. By the monotonicity of the normal cone operator, we have
\begin{align*}
\skal{\xi - \overline{\xi}, \h{x} - \overline{\h{x}}} \ge 0.
\end{align*}
By the strong convexity of ${F}_{\widetilde{\gamma}}$ on $S$, there exists $\widetilde{\mu} > 0$ such that
\begin{align*}
\skal{ \nabla{F}_{\widetilde{\gamma}}(\h{x}) - \nabla {F}_{\widetilde{\gamma}}(\overline{\h{x}}), \h{x} - \overline{\h{x}} } 
\ge \frac{\widetilde{\mu}}{2} \norm{\h{x} - \overline{\h{x}}}^2.
\end{align*}
Combining the two, we obtain
\begin{align*}
\norm{ \xi + \nabla {F}_{\widetilde{\gamma}}(\h{x}) } \norm{ \h{x} - \overline{\h{x}} } 
&\ge \skal{ \xi + \nabla {F}_{\widetilde{\gamma}}(\h{x}), \h{x} - \overline{\h{x}} } \\
&= \skal{ \xi - \overline{\xi}, \h{x} - \overline{\h{x}} } + \skal{ \nabla {F}_{\widetilde{\gamma}}(\h{x}) - \nabla {F}_{\widetilde{\gamma}}(\overline{\h{x}}), \h{x} - \overline{\h{x}} } \\
&\ge \frac{\widetilde{\mu}}{2} \norm{\h{x} - \overline{\h{x}}}^2.
\end{align*}
Therefore,
\begin{align}\label{KL_intermed}
\dist^2(\bm 0, \partial (F_{\widetilde{\gamma}}+\iota_S)(\h{x})) 
&= \inf_{\xi \in N_S(\h{x})} \norm{ \xi + \nabla {F}_{\widetilde{\gamma}}(\h{x}) }^2 \ge \frac{\widetilde{\mu}^2}{4} \norm{\h{x} - \overline{\h{x}}}^2.
\end{align}
Denoting by $L >0$ the Lipschitz constant of the gradient of ${F}_{\widetilde{\gamma}}$ on $S$, for every $\h x\in S$ it holds
\begin{align}\label{KL_intermed2}
F_{\widetilde{\gamma}}(\h{x}) - F_{\widetilde{\gamma}}(\overline{\h{x}}) 
\le \skal{ \nabla {F}_{\widetilde{\gamma}}(\overline{\h{x}}), \h{x} - \overline{\h{x}} } + \frac{L}{2} \norm{ \h{x} - \overline{\h{x}} }^2.
\end{align}

If $\overline{\delta} = 0$, then $\nabla {F}_{\widetilde{\gamma}}(\overline{\h{x}}) = 0$ and the KL property follows directly from the two inequalities above.

If $\overline{\delta} > 0$, then there exists a neighborhood $U$ of $\overline{\h{x}}$ such that ${F}_{\widetilde{\gamma}}$ is differentiable on $U$ and $\norm{ \nabla {F}_{\widetilde{\gamma}}({\h{x}}) } > \epsilon$ for some $\epsilon > 0$ and every $\h{x} \in U$. Then, for every $\h{x} \in U \cap \mathrm{int}(S)$
\[
\dist^2(0, \partial (F_{\widetilde{\gamma}}+\iota_S)(\h{x})) = \norm{ \nabla {F}_{\widetilde{\gamma}}(\h{x}) }^2 > \epsilon^2.
\]
For $\h{x} \in U \cap \mathrm{bd}(S)$, note that $\norm{ \overline{\h{x}} - \h{u} }^2 = \norm{ \h{x} - \h{u} }^2$, thus
\begin{align*}
\skal{\nabla {F}_{\widetilde{\gamma}}(\overline{\h{x}}), \h{x} - \overline{\h{x}} } 
&= \delta \skal{ -(\overline{\h{x}} - \h{u}), \h{x} - \overline{\h{x}} } \\
&= \frac{\delta}{2} \norm{ \overline{\h{x}} - \h{u} }^2 + \frac{\delta}{2} \norm{ \overline{\h{x}} - \h{x} }^2 - \frac{\delta}{2} \norm{ \h{x} - \h{u} }^2 \\
&= \frac{\delta}{2} \norm{ \overline{\h{x}} - \h{x} }^2.
\end{align*}

Thus, by \eqref{KL_intermed} and \eqref{KL_intermed2}, we can choose an open neighborhood $V \subseteq U$ of $\overline{\h x}$ such that for every $\h x \in V \cap S$ it holds
\begin{align*}
\frac{2(L + \delta)}{\widetilde{\mu}^2} \dist^2(0, \partial (F_{\widetilde{\gamma}}+\iota_S)(\h{x})) 
\ge \frac{L + \delta}{2} \norm{ \overline{\h{x}} - \h{x} }^2 
\ge F_{\widetilde{\gamma}}(\h{x}) - F_{\widetilde{\gamma}}(\overline{\h{x}}).
\end{align*}
This concludes the proof.
\end{proof}

\subsection{Numerical results} All experiments were carried out on audio sequences from the test set \textcopyright \texttt{EBU SQAM CD}, specifically {the} signals numbered $39$--$70$, which include instruments, orchestra, music, and speech. To ensure consistent duration and comparable signal energy, we extracted a 2-second segment $\overline{\h x}$ from each audio file, namely the interval $t = 1\text{s}$ to $t = 3\text{s}$. White Gaussian noise with mean 0 was then added to each signal {with a standard deviation set to $\sigma= \frac{1}{20\sqrt{L}}\norm{\overline{\h x}}$, where $L$ is the length of the sampled signal $\overline{\h x}$}.

For the STFT operator $\mathcal{T}$, we used a Gaussian window of length $2048$ and a hop size of $512$. This resulted in a redundancy factor of $4$, so that $\sigma_{\mathcal{T}}=\norm{\mathcal{T}}^2=4$, and the injectivity constant is $s=1$ (i.e., $\norm{\mathcal{T}\,\h {x}} \ge s\,\norm{\h {x}}$ for every $\h {x}$, see \cite{Feichtinger98}). As a quality measure for the reconstruction, we used the \textit{improvement in signal-to-noise ratio} (ISNR):
\begin{align*}
    \text{ISNR}(\h {x}^k) 
    &= 10 \log_{10} \left(\frac{\norm{\overline{\h {x}} - \h {u}}_2^2}{\norm{\overline{\h {x}} - \h {x}^k}_2^2}\right),
\end{align*}
where $\overline{\h {x}}$ denotes the unknown original signal, $\h {u}$ the noisy signal, and $\h {x}^k$ the reconstruction after $k$ iterations. We set $\lambda := 10^{-1}$ in (\ref{eq:num_KL_objective}) and $p := \tfrac12$. 

Since $g$ has full domain, Assumption~\ref{ass:A(S)} holds for any $r>0$. 
The Lipschitz constant of $g$ on the set $\mathcal{K}_r$ {as defined in Assumption \ref{ass:A(S)}} is given by
\[
\ell_{g,\mathcal{K}_r} \;=\; \sqrt{N} \;+\; \lambda \max_{\h {z}\,\in\,\mathcal{K}_r}\!\norm{\h {z}} = \sqrt{N} \;+\; \lambda \big(\norm{T}\,(1+p)\,\norm{\h {u}} \;+\; r\big).
\]
As $r \to +\infty$, we have $\tfrac{r}{\ell_{g,\mathcal{K}_r}} \to \tfrac{1}{\lambda}$, so every $0<\gamma<\tfrac{2}{\lambda}$ satisfies the assumption in Theorem~\ref{Algo1:conv} for sufficiently large $r$. Furthermore, ${\h z_0}$ can be also chosen arbitrary. We also have
\[
\mu \;=\; \inf_{\h {x} \,\in\, S}\!\norm{\mathcal{T}\,\h {x}} 
\;\ge\; s\,(1-p)\,\norm{\h {u}}
\;=\; \tfrac12\,\norm{\h {u}},
\]
so the condition in Theorem~\ref{KL1o2} becomes $r_1 > \tfrac{8\,r_2}{\norm{\h {u}}} + 4\lambda$. 

In our test setup, we have $\norm{\h {u}} \ge 2$ for every ${\h u}$ and every $r_1$ and $r_2$ fulfilling $r_1 > 4r_2 + 4\lambda$, hence the strong convergence guarantee holds.

We used signal 39 from the test set to analyze the effect of the parameters on the iterative scheme. We considered the parameters $\beta_k \equiv 1$ for all $k \geq 0$, $\rho_0 = 1$, $\mu = 1$, and $\varepsilon = 10^{-10}$, as preliminary tests showed that these parameters had little effect on long-term convergence or reconstruction quality. We used the explicit step size $$\alpha_k := \alpha - \frac{\gamma}{\sqrt{k+1}} \quad \forall k \geq 0.$$

We first investigated which values of $r_1$ and $r_2$ give the best reconstruction. For $\gamma=0.1$ and $\alpha = \gamma/64$ we show the resulting ISNR of the algorithm \ref{AdaptiveDPFS} after $1000$ iterations for different values of $r_1$ and $r_2$ in  Table \ref{tab:r1r2}.

\begin{table}[htbp]
\centering
\renewcommand{\arraystretch}{1.5} 
\setlength{\tabcolsep}{8pt} 
\begin{tabular}{|c|c|c|c|c|c|}
\hline
ISNR & $r_2=0$ & $r_2=0.2$ & $r_2=0.4$  & $r_2=0.6$ & $r_2=0.8$ \\ \hline
$r_1 = 4$ & 2.473 & 2.667 & 2.801 & 2.868 & 2.871 \\
$r_1 = 8$ & 3.532 & 3.663 & 3.750 & 3.792 & 3.792 \\
$r_1 = 9$ & 3.603 & 3.714 & 3.787 & \textbf{3.821} & 3.820 \\
$r_1 = 10$ & 3.582 & 3.673 & 3.732 & 3.759 & 3.756 \\
$r_1 = 11$ & 3.489 & 3.561 & 3.607 & 3.627 & 3.624 \\
$r_1 = 12$ & 3.349 & 3.406 & 3.441 & 3.456 & 3.452 \\
$r_1 = 16$ & 2.685 & 2.706 & 2.719 & 2.723 & 2.719 \\
$r_1 = 24$ & 1.774 & 1.778 & 1.780 & 1.779 & 1.776 \\\hline
\end{tabular}
\caption{\textit{ISNR for Adaptive DPFS after $1000$ iterations for different values of $r_1$ and $r_2$.}}
\label{tab:r1r2}
\end{table}

The case $r_2 = 0$ corresponds to standard convex $\norm{\cdot}_1$-regularization, while $r_2 > 0$ introduces a difference-of-convex (DC) structure. {We observe a clear gain in ISNR, whenever $r_2$ is chosen positive}. This indicates that the DC regularizer improves the reconstruction quality.

For $r_1=9$ and $r_2=0.6$, we plot the resulting ISNR for different values of $\alpha$ and $\gamma$, where $\alpha <\frac{\gamma}{2}$, in Table \ref{tab:alphagamma}. {For fixed $\gamma$, the algorithm consistently converged to the same reconstruction, largely independent of $\alpha$, with only minimal speed-up for smaller $\alpha$. As $\alpha$ is merely a lower bound on the decreasing sequence $(\alpha_k)_{k \geq 0}$, its value has little impact on performance. In contrast, $\gamma$ affects $\alpha_k - \alpha = \frac{\gamma}{\sqrt{k+1}}$ and thus influences the condition in line~5 of Algorithm~\ref{AdaptiveDPFS}, highlighting its greater relevance.}

\begin{table}[htbp]
\centering
\renewcommand{\arraystretch}{1.5} 
\setlength{\tabcolsep}{8pt} 
\begin{tabular}{|c|c|c|c|c|c|}
\hline
ISNR & $\alpha = \gamma / 64$ & $\alpha = \gamma / 32$ & $\alpha = \gamma / 16$ & $\alpha = \gamma /8$ &$\alpha = \gamma /4$   \\ \hline
$\gamma=$0.05 & 3.532561 & 3.532414 & 3.532100 & 3.531389 & 3.529637 \\
$\gamma=$0.1 & \textbf{3.823013} & 3.822894 & 3.822646 & 3.822101 & 3.820805 \\
$\gamma=$0.15 & 3.788634 & 3.788586 & 3.788489 & 3.788288 & 3.787863 \\
$\gamma=$0.2 & 3.765046 & 3.765021 & 3.764972 & 3.764877 & 3.764695 \\
 \hline
\end{tabular}
\caption{\textit{ISNR for Adaptive DPFS after $1000$ iterations for different values of $\gamma$ and $\alpha$.}}
\label{tab:alphagamma}
\end{table}

In Figure~\ref{fig:convergence_plot} we plot the distance between the iterates of Adaptive DPFS and their limiting critical point to illustrate the linear convergence rate of the algorithm.

\begin{figure}[htbp]
    \centering
    \includegraphics[width=0.8\linewidth]{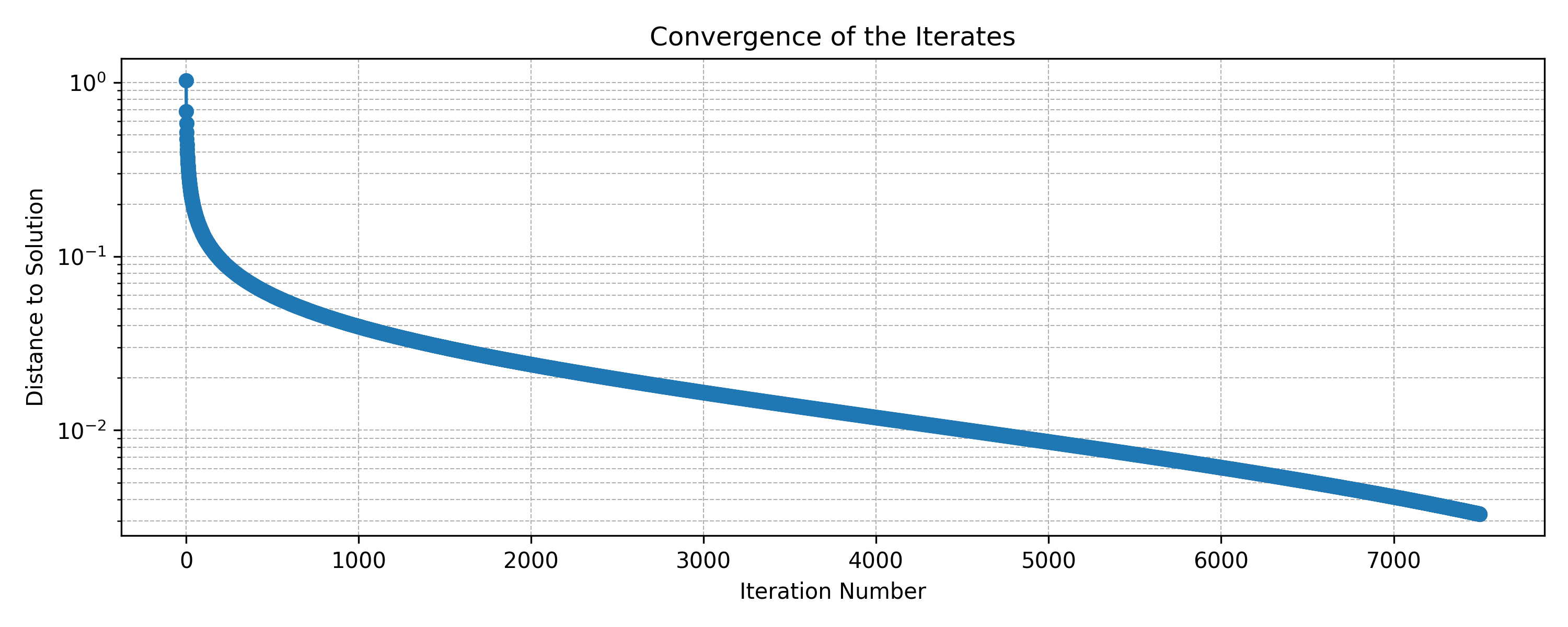}
    \caption{Convergence of the iterates ${\h x}^k$ generated by Adaptive DPFS to their limit $\widebar{\h x}$ expressed in terms of $\norm{{\h x}^k - \widebar{\h x}}$.}
    \label{fig:convergence_plot}
\end{figure}

\begin{algorithm}
\caption{Adaptive DPFS with line search}
\label{A-FSPS-MLS}
\begin{algorithmic}[1]
    \Statex{\textbf{given:} $\ell_{\nabla \varphi}>0$, $\sigma_{A}$.}
\Statex{\textbf{choose:} $\beta_0>0$,  $\rho_0>0$,   $0<\alpha<\frac{\gamma}{2}$, $\alpha_0 = \alpha + \gamma$,
			$\delta_0 = \frac{\alpha}{\alpha_0}$, $\mu >0$, $\varepsilon>0$, $\eta > 1$, $0<\nu<1,\; c>0$, $T\in {\mathbb N}$, $t \in{\mathbb N}$, and a starting point $({\h x}^0, {\h y}^0, {\h z}^0)$}
\For{$k = 0, 1, 2, ...$}		
\State{${\h y}^{k+1} := \prox_{f^*,1/\beta_0}({\h y}^k + \frac{1}{\beta_0}K{\h x}^{k})$,}
		\State{${\h z}^{k+1} := \prox_{g^*,1/\alpha_k}(\delta_k{\h z}^k + \frac{1}{\alpha_k}A{\h x}^{k})$.}
		\If{$(\alpha_k - \alpha)\|{\h z}^{k+1}\| > \varepsilon$}
				\State{{Update {$\alpha_{k+1}$ according to \eqref{scalar} and \eqref{scalar2},}}}
				\State{Update $\rho_{k+1} := {\frac{\ell_{\nabla  \varphi}}{2}} + \frac{3\gamma\sigma_{A}}{(\alpha_{k+1} - \alpha)^2} + \mu$.}
				\State{Update $\delta_{k+1} := \frac{\alpha}{\alpha_{k+1}}$.}
       \Else
           \State{$\alpha_{k+1}:=\alpha_k$; $\rho_{k+1}:=\rho_k$; $\delta_{k+1}:=\delta_k$.}
	    \EndIf
		\State{$\h d^{k+1} := -\nabla \varphi({\h x}^k) + K^* {\h y}^{k+1} - A^* {\h z}^{k+1}.$}
        \State{$\rho_{k+1,0}:=\rho_{k+1}$.}
		\For{$j = 0 : t-1$}
		\State{$\rho_{k+1,j} := \nu\eta^{j}\rho_{k+1,0}.$}
		\State{$ \widehat{\h x}^{k+1}:={ \text{Proj}_S}\left({\h x}^k + \frac{1}{\rho_{k+1,j}}{\h d}^{k+1}\right).$}
		\If{$F(\widehat{\h x}^{k+1})\leq \max\limits_{[k-T]_{+}\leq t\leq k}F(\h x^{t}) - \frac{c}{2}||\h x^{k}-\widehat{\h x}^{k+1}||^{2}$}
		\State $\h x^{k+1} := \widehat{\h x}^{k+1}.$
        \State $\rho_{k+1}:=\rho_{k+1,j}$.
		\State \textbf{break}
		\EndIf
		\EndFor	
		\EndFor
	\end{algorithmic}
\end{algorithm}

We also introduced an enhanced version of Adaptive DPFS by integrating a nonmonotone line search technique \cite{WNF09}, see Algorithm \ref{A-FSPS-MLS}. This addition was designed to improve the overall performance of Adaptive DPFS. For simplicity, we considered $\beta_k\equiv \beta_0$ for every $k \geq 0$, but the parameter $\beta_k$ could be adjusted at each iteration.

We initialized Algorithm~\ref{A-FSPS-MLS} with $\beta_0 = 1$, $\rho_0 = 1$, $\mu = 1$, and $\varepsilon = 10^{-10}$. We further set $c = 10^{-3}$, $T = 1$, and $t = 5$.
Table~\ref{table:ETA_NU_isnr} reports the ISNR values achieved for various choices of $\nu$ and $\eta$, with $\gamma = 0.1$, $\alpha = \gamma/64$, $r_1 = 9$, and $r_2 = 0.6$. The nonmonotone line search technique yields slightly better results than Algorithm~\ref{AdaptiveDPFS}  on fixed iteration count when $\nu$ is close to 0 or 1, and in particular accelerates the convergence in the early iterations. The parameter $\eta$ has no significant influence on the overall performance.

\begin{table}[htbp]
\centering
\renewcommand{\arraystretch}{1.5} 
\setlength{\tabcolsep}{8pt} 
\begin{tabular}{|c|c|c|c|c|}
\hline
ISNR &  $\eta = 1.1$ & $\eta = 1.3$ & $\eta = 1.5$ & $\eta = 1.7$  \\
\hline
$\nu = 0.1$ & 3.831794 & 3.831795 & 3.831794 & 3.831794 \\

$\nu = 0.2$ & 3.827517 & 3.827460  & 3.827415 & 3.827495 \\

$\nu = 0.3$ & 3.822379 & 3.822576  & 3.822944 & 3.823497  \\

$\nu = 0.4$ & 3.818764 & 3.820110 & 3.820616 & 3.820884  \\

$\nu = 0.5$ & 3.817944 & 3.817945 & 3.817944 & 3.817944 \\

$\nu = 0.6$ & 3.820944 & 3.820944 & 3.820944 & 3.820944 \\

$\nu = 0.7$ & 3.824268 & 3.824268 & 3.824268 & 3.824268  \\

$\nu = 0.8$ & 3.824263 & 3.824264 & 3.824263 & 3.824263 \\

$\nu = 0.9$ & 3.817724 & 3.817725 & 3.817724 & 3.817724 \\
\hline

\end{tabular}
\caption{ISNR of Adaptive DPFS after $1000$ iterations for different values of $\eta$ and $\nu$.}
\label{table:ETA_NU_isnr}
\end{table}

For the purpose of comparisons, we considered in the context of solving the DC program \eqref{eq:num_KL_objective} also the double proximal-gradient (FBDC) method of Banert and Bo\c{t} \cite{SB19}, which generates the sequence of iterates for all $k \geq 0$ as follows:
\begin{equation} \label{method:BaBo}
\begin{aligned}
    \h x^{k+1} &:= \prox_{\gamma_k g\circ \mathcal{T}} \left( \h x^k + \gamma_k \mathcal{T}^* \h y^k - \gamma_k \nabla \varphi (\h x^k) \right), \\
    \h y^{k+1} &:= \prox _{\mu_k f^*} \left( \h y^k + \mu_k \mathcal{T} \h x^{k+1} \right).
\end{aligned}
\end{equation}

In \cite{SB19} they applied this method to a DC program of the form \eqref{eq:num_KL_objective}. In the absence of a closed formula for $\prox_{g\circ {\cal T}}$, they used an inner-loop solver with a stopping criterion, which we also considered to perform the comparison. Based on preliminary tests, we chose $\gamma_k \equiv 0.025$ and $\mu_k \equiv 0.5$ for all $k \geq 0$, {as the best performing parameter combination for FBDC in this setting.}

We carried out the  comparison of Algorithm~\ref{AdaptiveDPFS}, Algorithm~\ref{A-FSPS-MLS}, and FBDC method ~\eqref{method:BaBo} on the signals 40–70 of the test set mentioned above. Figure~\ref{fig:enter-label} illustrates the evolution of the ISNR with respect to the computational time. Table~\ref{tab:isnr-results} reports the average ISNR achieved, the quartiles, and the average number of iterations performed within the respective computational time.

\begin{figure}
    \centering
    \includegraphics[width=0.5\linewidth]{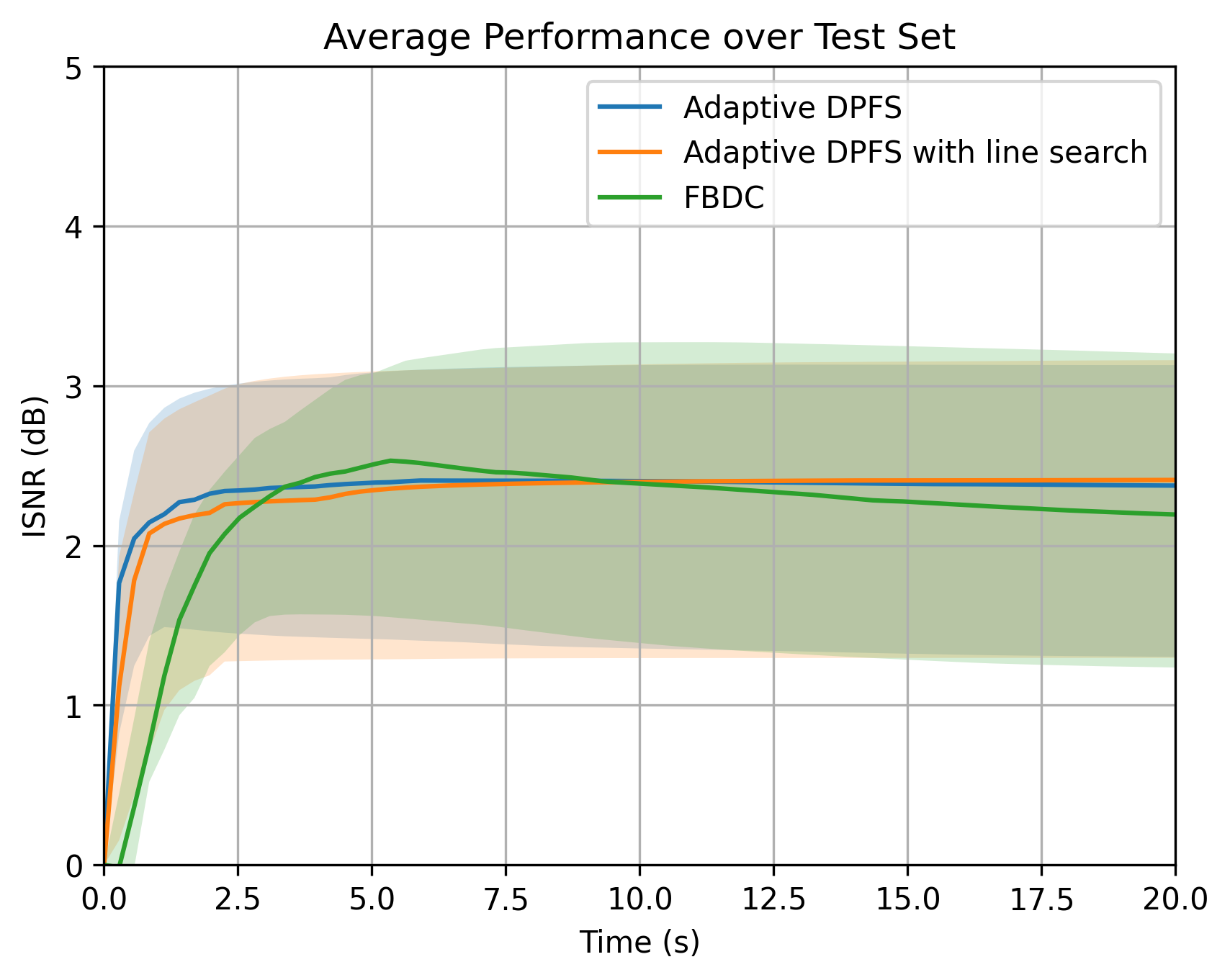}
    \caption{Comparison of the ISNR achieved by Adaptive DFPS, Adaptive DFPS with line search, and the FBDC method.}
    \label{fig:enter-label}
\end{figure}
\begin{table}
    \centering
\begin{tabular}{l|lccc}
    \textbf{Time} & \textbf{Method} & \textbf{ISNR (dB)} & \textbf{25\% / 75\%} & \textbf{Iter.} \\
    \hline
    \multirow{3}{*}{30s} & Adaptive DPFS     & 2.52 & 1.46 / 3.19 & 363 \\
                          & + Line Search     & 2.46 & 1.31 / 3.15 & 270.2 \\
                          & FBDC              & 2.35 & 1.22 / 3.13 & 21.9  \\
    \hline
    \multirow{3}{*}{100s} & Adaptive DPFS     & 2.52 & 1.39 / 3.21 & 1302 \\
                          & + Line Search     & 2.48 & 1.29 / 3.18 & 968  \\
                          & FBDC              & 2.34 & 1.19 / 3.09 & 63.5  \\
\end{tabular}

        \caption{Average ISNR, quartiles and the average iteration count in the respective computational time.}
        \label{tab:isnr-results}
\end{table}

\appendix

\end{document}